\documentclass{article}%
\usepackage{amsmath}
\usepackage{amsfonts}
\usepackage{amssymb}
\usepackage{graphicx}%
\providecommand{\U}[1]{\protect\rule{.1in}{.1in}}
\newtheorem{theorem}{Theorem}

\newtheorem{corollary}[theorem]{Corollary}

\newtheorem{definition}[theorem]{Definition}

\newtheorem{lemma}[theorem]{Lemma}

\newtheorem{proposition}[theorem]{Proposition}
\newtheorem{remark}[theorem]{Remark}

\newenvironment{proof}[1][Proof]{\noindent\textbf{#1.} }{\ \rule{0.5em}{0.5em}}
\begin{document}

\title{Existence and characterization of attractors for a nonlocal reaction-diffusion
equation having an energy functional}
\author{R. Caballero$^{1}$, P. Mar\'{\i}n-Rubio$^{2}$ and Jos\'{e} Valero$^{1}$\\$^{1}${\small Centro de Investiagaci\'{o}n Operativa, Universidad Miguel
Hern\'{a}ndez de Elche,}\\{\small Avda. Universidad s/n, 03202, Elche (Alicante), Spain}\\$^{2}${\small Dpto. Ecuaciones Diferenciales y An\'{a}lisis Num\'{e}rico,}\\{\small Universidad de Sevilla, Apdo. de Correos 1160, 41080-Sevilla, Spain}}
\date{}
\maketitle

\begin{abstract}
In this paper we study a nonlocal reaction-diffusion equation in which the
diffusion depends on the gradient of the solution.

We prove first the existence and uniqueness of regular and strong solutions.
Second, we obtain the existence of global attractors in both situations under
rather weak assumptions by the defining a multivaled semiflow (which is a
semigroup in the particular situation when uniqueness of the Cauchy problem is
satisfied). Third, we characterize the attractor either as the unstable
manifold of the set of stationary points or as the stable one when we consider
solutions only in the set of bounded complete trajectories.

\end{abstract}

\bigskip



\section{Introduction}

In real applications usually we do not have enough information about the
systems under study and its features at every point. In reality, the
measurements are not made pointwise but through some local average. Therefore,
here arises the importance of nonlocal models. During the last decades many
mathematicians have been studying nonlocal problems motivated by its various
applications in physics, biology or population dynamics \cite{d2,
ChipotLovat97,d3, d41, d5, d7}.

For instance, let consider the problem of finding a function $u(t,x)$ such
that
\begin{equation}
\left\{
\begin{array}
[c]{l}%
u_{t}-a(\int_{\Omega}u(t,x)dx)\Delta u=g(t,u),\ \text{in }\Omega
\times(0,\infty),\\
u=0\quad\text{on }\partial\Omega\times(0,\infty),\\
u(0)=u_{0}\quad\text{in }\Omega.
\end{array}
\right.  \label{Problem1}%
\end{equation}
Here $\Omega$ is a bounded open subset in $\mathbb{R}^{n}$, $n\geq1,$ with
smooth boundary and $a$ is some function from $\mathbb{R}$ onto $(0,+\infty)$.
In such equation $u$ could describe the density of a population subject to
spreading. The diffusion coefficient $a$ is then supposed to depend on the
entire population in the domain rather than on the local density.

A wide literature with significant results about (\ref{Problem1}) have been
developed during the last few decades (see for example \cite{ChipotLovat97,d5,
d7}). However, it is possible to distinguish two basic cases of the following
more general equation
\[
\left\{
\begin{array}
[c]{l}%
u_{t}-a(u)\Delta u=g(t,u),\quad t>0,\ x\in\Omega,\\
u=0,\text{ \textit{in }}\partial\Omega\times\left(  0,\infty\right)  ,\\
u(0,x)=u_{0}(x)\quad x\in\Omega.
\end{array}
\right.
\]

Some authors consider $a$ depending on a linear functional $l(u)$, i.e.,
\[
a(u)=a(l(u))
\]
with
\[
l(u)=\int_{\Omega}g(x)u(x,t)dx,
\]
where $g(x)$ is a given function in $L^{2}(\Omega)$. For $g(t,u)=f(t)$ the
existence and uniqueness of solutions and their asymptotic behavior are
studied for example in \cite{d3,d41,d4,ZhengChipot}. For $g(t,u)=f(u)+h(t)$
the existence, uniqueness and asymptotic behaviour of solutions is studied in
\cite{CaHeMa15,d0,CaHeMa18}. Moreover, the authors prove the existence of
pullback attractors in $L^{2}(\Omega)$ and $H_{0}^{1}(\Omega)$. Extensions in
this direction for equations governed by the p-laplacian operator instead of
the laplacian operator $\Delta$ are given in \cite{CAHeMa17,CaHeMa18B}.

On the other hand, it is possible to consider a function $a$ such that
$a\left(  u\right)  =a(\Vert u\Vert_{H_{0}^{1}}^{2})$. The existence and
uniqueness of solutions of the following problem
\[
\left\{
\begin{array}
[c]{l}%
u_{t}-a(\Vert u\Vert_{H_{0}^{1}}^{2})\Delta u=f,\quad t>0,\ x\in\Omega,\\
u=0,\text{ in }\partial\Omega\times\left(  0,\infty\right)  ,\\
u(0,x)=u_{0}(x)\quad x\in\Omega.
\end{array}
\right.
\]
is proved in \cite{ZhengChipot,ChipotValente}, where $f\in L^{2}(\Omega),$
$u_{0}\in H_{0}^{1}(\Omega)$ and $a=a(s)$ is a continuous function such that
$0<m\leq a(s)\leq M.$

By this way, in our paper the following problem is considered
\begin{equation}
\left\{
\begin{array}
[c]{l}%
u_{t}-a(\Vert u\Vert_{H_{0}^{1}}^{2})\Delta u=f(u)+h(t),\ \text{in }%
\Omega\times(0,\infty),\\
u=0\quad\text{on }\partial\Omega\times(0,\infty),\\
u(0,x)=u_{0}\left(  x\right)  \quad\text{in }\Omega,
\end{array}
\right.  \label{21}%
\end{equation}
where $h(t)\in L^{2}(0,T;L^{2}(\Omega)),$ for all $T>0,$ $a:\mathbb{R}%
^{+}\rightarrow\mathbb{R}^{+}$ is a continuous function such that $a\left(
s\right)  \geq m>0$ and $f$ is a continuous function satisfying standard
dissipative and growth conditions (see (\ref{3}) below).

The aim of this paper is three-fold. First, we will prove the existence of
solutions for problem (\ref{21}) under different assumptions on the nonlinear
function $f$. Second, we will obtain the existence of attractors for the
semiflows generated by either regular or strong solutions. Third, we establish
that the global attractor can be characterized by the unstable manifold of the
set of stationary points. It is important to notice that the proof of this
last fact requires the existence of a Lyapunov function on the attractor, and
for this aim the term $a(\Vert u\Vert_{H_{0}^{1}}^{2})$ is crucial. In the
case when $a(u)=a(l(u))$ it is not known whether such a function exists or not.

We prove the existence of strong solutions by assuming that either the
function $f$ is continuously differentiable and $f^{\prime}\left(  s\right)
\leq\eta$ or a more strict growth condition on $f.$ Supposing additionaly that
the function $a$ has sublinear growth we prove the existence of regular
solutions as well. Moreover, when $f^{\prime}\left(  s\right)  \leq\eta$ and
the function $a\left(  s^{2}\right)  s$ is non-decreasing, uniqueness is proved.

When studying the asymptotic behaviour of solutions, new challenging
difficulties arise for problem (\ref{21}). For this problem we consider the
autonomous situation, that is, $h\in L^{2}\left(  \Omega\right)  $ does not
depend on $t$.

If uniqueness holds, then we define classical semigroups (one for regular
solutions and one for strong solutions) and prove the existence of the global
attractor. Under some extra assumptions on the functions $a,h$ we are able to
obtain the the global attractor is bounded in $H^{2}\left(  \Omega\right)  $
and $L^{\infty}\left(  \Omega\right)  $. With this regularity at hand we
define a Lyapunov function in the attractor which allows us to study its
structure and characterize it as the unstable manifold of the set of
stationary points (denoted by $M^{u}\left(  \mathfrak{R}\right)  $). Also, the
attractor is equal to the stable set of the stationary points when we consider
solutions only in the set of bounded complete trajectories (denoted by
$M^{s}\left(  \mathfrak{R}\right)  $).

If uniqueness is not known to be true, then we have to define a (possibly)
multivalued semiflow. Then the existence of the global attractor is proved for
regular solutions in the topology of the space $L^{2}\left(  \Omega\right)  $
and for strong solutions in the topology of the space $H_{0}^{1}\left(
\Omega\right)  $, extending in this way the known results for the local
problem \cite{kapustyankasyanov}.

The structure of the global attractor is an important feature as it gives us
an insight into the long-term dynamics of the solutions. In the multivalued
situation it is a challenging problem that has not been completely understood
yet. So far in the local case several results in this direction have been
obtained for reaction-diffusion equations without uniqueness
\cite{ARV06,CabCarvMarVal,kapustyankasyanov,KKV15}.

In our nonlocal problem for both situations (for regular and strong solutions)
we are able under some conditions to define a Lyapunov function on the
attractor and to prove that it is characterized as in the single-valued case
by
\[
\mathcal{A}=M^{u}(\mathfrak{R})=M^{s}(\mathfrak{R}).
\]

\section{Existence of solutions}

Throughout this paper we will denote by $\left\Vert \text{\textperiodcentered
}\right\Vert _{X}$ the norm in the Banach space $X.$

We consider the following nonlocal reaction-diffusion equation
\begin{equation}
\left\{
\begin{array}
[c]{l}%
u_{t}-a(\Vert u\Vert_{H_{0}^{1}}^{2})\Delta u=f(u)+h(t),\ \text{in }%
\Omega\times(0,\infty),\\
u=0\quad\text{in }\partial\Omega\times(0,\infty),\\
u(0,x)=u_{0}(x)\quad\text{in }\Omega,
\end{array}
\right.  \label{1}%
\end{equation}
where $\Omega$ is a bounded open set of $\mathbb{R}^{n}$ with smooth boundary
$\partial\Omega.$

Let us consider the following conditions on the functions $a,f,h:$%
\begin{equation}
h\in L^{2}(0,T;L^{2}(\Omega))\text{ }\forall T>0, \label{h}%
\end{equation}%
\begin{equation}
a\in C(\mathbb{R}^{+},\mathbb{R}^{+}),\ f\in C(\mathbb{R},\mathbb{R}),
\label{Cont}%
\end{equation}%
\begin{equation}
a\left(  s\right)  \geq m>0, \label{2}%
\end{equation}%
\begin{equation}
-\kappa-\alpha_{2}|s|^{p}\leq f(s)s\leq\kappa-\alpha_{1}|s|^{p}, \label{3}%
\end{equation}
where $m,\ \alpha_{1},\ \alpha_{2}>0$ and $\kappa\geq0,$ $p\geq2$. Observe
that then there exists $C>0$ such that
\begin{equation}
|f(s)|\leq C(1+|s|^{p-1})\quad\forall s\in\mathbb{R}, \label{1.4}%
\end{equation}
and that the function $\mathcal{F}(s):=\int_{0}^{s}f(r)dr$ satisfies
\begin{equation}
-\widetilde{\alpha}_{2}|s|^{p}-\widetilde{\kappa}\leq\mathcal{F}%
(s)\leq\widetilde{\kappa}-\widetilde{\alpha}_{1}|s|^{p} \label{4}%
\end{equation}
for certain positive constants $\widetilde{\alpha}_{i},$ $i=1,2,$ and
$\widetilde{\kappa}\geq0,$ and
\begin{equation}
|\mathcal{F}(s)|\leq\widetilde{C}(1+|s|^{p})\quad\forall s\in\mathbb{R}.
\label{5}%
\end{equation}

\begin{definition}
A weak solution to (\ref{1}) is a function $u\left(  \text{\textperiodcentered
}\right)  $ such that $u\in L^{\infty}(0,T;L^{2}(\Omega))\cap L^{2}%
(0,T;H_{0}^{1}(\Omega))\cap L^{p}(0,T;L^{p}(\Omega))$ for any $T>0$ and
satisfies the equality%
\begin{equation}
\frac{d}{dt}(u,v)+a(\Vert u(t)\Vert_{H_{0}^{1}}^{2})(\nabla u(t),\nabla
v)=(f(u(t)),v)+(h(t),v)\quad\forall v\in H_{0}^{1}(\Omega)\cap L^{p}(\Omega),
\label{equationweak}%
\end{equation}
in the sense of scalar distributions.
\end{definition}

Here, we denote by $(\cdot,\cdot)$ the inner product in $L^{2}(\Omega)$ (or
$\left(  L^{2}(\Omega)\right)  ^{d}$ for $d\in\mathbb{N}$) and also the
duality product between $L^{p}(\Omega)$ and $L^{q}(\Omega)$ (where $q$ is the
conjugate exponent of $p$, that is, $q=p/(p-1)$). The duality between
$H_{0}^{1}\left(  \Omega\right)  $ and $H^{-1}\left(  \Omega\right)  $ will be
denoted by $\left\langle \text{\textperiodcentered,\textperiodcentered
}\right\rangle .$

We need to guarantee that the initial condition of the problem makes sense for
a weak solution. This can be achieved in a standard way assuming that the
function $a$ has an upper bound, that is, there exists $M>0$ such that
\begin{equation}
a\left(  s\right)  \leq M\text{ for all }s\geq0. \label{6}%
\end{equation}
Indeed, if $u$ is a weak solution to (\ref{1}), taking into account
(\ref{1.4}) and (\ref{6}) it follows that
\begin{equation}
u_{t}=a(\Vert u\Vert_{H_{0}^{1}}^{2})\Delta u+f(u)+h\in L^{2}(0,T;H^{-1}%
(\Omega))+L^{q}(0,T;L^{q}(\Omega))\subset L^{q}(0,T;H^{-s}(\Omega)),
\label{6b}%
\end{equation}
for $s\geq n(\frac{1}{q}-\frac{1}{2}).$ Therefore, by \cite[p.283]{chepvishik}
$u\in C([0,T],L^{2}(\Omega))$ and the initial condition makes sense when
$u_{0}\in L^{2}(\Omega)$.

For the operator $A=-\Delta$, thanks thanks to the assumptions made on the
domain $\Omega$, it is well known that $D(A)=H^{2}(\Omega)\cap H_{0}%
^{1}(\Omega))$ \cite[Proposition 6.19]{robinson}.

\begin{definition}
A regular solution to (\ref{1}) is a weak solution with the extra regularity
$u\in L^{\infty}(\varepsilon,T;H_{0}^{1}(\Omega))$ and $u\in L^{2}%
(\varepsilon,T;D(A))$ for any $0<\varepsilon<T.$
\end{definition}

\begin{remark}
Since $\dfrac{du}{dt}\in L^{q}\left(  \varepsilon,T;L^{q}\left(
\Omega\right)  \right)  $ for any regular solution, in this case equality
(\ref{equationweak}) is equivalent to the following one:%
\begin{align}
&  \int_{\varepsilon}^{T}\int_{\Omega}\frac{du\left(  t,x\right)  }{dt}%
\xi\left(  t,x\right)  dxdt+\int_{\varepsilon}^{T}a(\Vert u(t)\Vert_{H_{0}%
^{1}}^{2})\int_{\Omega}\nabla u\text{\textperiodcentered}\nabla\xi
dxdt\label{EquationRegular}\\
&  =\int_{\varepsilon}^{T}\int_{\Omega}f\left(  u\left(  t,x\right)  \right)
\xi\left(  t,x\right)  dxdt+\int_{\varepsilon}^{T}\int_{\Omega}h\left(
t,x\right)  \acute{\xi}\left(  t,x\right)  dxdt,\nonumber
\end{align}
for all $0<\varepsilon<T$ and $\xi\in L^{p}\left(  0,T;L^{p}\left(
\Omega\right)  \right)  .$
\end{remark}

\begin{lemma}
\label{DerivComposition}Let $u\in L^{p}\left(  \varepsilon,T;X\right)  $,
$\dfrac{du}{dt}\in L^{q}\left(  \varepsilon,T;X^{\ast}\right)  $ for all
$0<\varepsilon<T$, where $X$ is a reflexive and separable Banach space and
$X^{\ast}$ denotes its dual space. Assume that $\beta\in W^{1,\infty
}(\mathbb{\varepsilon},T;[\alpha\left(  \varepsilon\right)  ,\alpha\left(
T\right)  ])$ and $0<\beta\left(  \varepsilon\right)  <\beta\left(  T\right)
$ for all $0<\varepsilon<T$. Then $w\left(  \text{\textperiodcentered}\right)
=u\left(  \beta\left(  \text{\textperiodcentered}\right)  \right)  \in
L^{p}\left(  \varepsilon,T;X\right)  $, $\dfrac{dw}{dt}\in L^{q}\left(
\varepsilon,T;X^{\ast}\right)  ,$ for all $0<\varepsilon<T$, and%
\begin{equation}
\dfrac{dw}{dt}\left(  t\right)  =\frac{du}{dt}\left(  \beta\left(  t\right)
\right)  \frac{d\beta}{dt}\left(  t\right)  \text{ for a.a. }t>0.
\label{EqualityDeriv}%
\end{equation}

\end{lemma}

\begin{proof}
We fix arbitrary $0<\varepsilon<T$. There exists a sequence $u_{n}\in
C^{1}\left(  [\beta\left(  \varepsilon\right)  ,\beta(T)],X\right)  $ such
that $u_{n}\rightarrow u$ in $L^{p}\left(  \beta\left(  \varepsilon\right)
,\beta(T);X\right)  $ and $\dfrac{du_{n}}{dt}\rightarrow\dfrac{du}{dt}$ in
$L^{q}\left(  \beta\left(  \varepsilon\right)  ,\beta(T);X^{\ast}\right)  $
\cite[Chapter IV]{Gajewski}. We define $w_{n}\left(  t\right)  =u_{n}\left(
\beta\left(  t\right)  \right)  $. Following the same proof of \cite[Corollary
VIII.10]{Brezis} we obtain that $w_{n}\left(  \text{\textperiodcentered
}\right)  \in W^{1,\infty}\left(  \varepsilon,T;X\right)  $ and
\[
\frac{dw_{n}}{dt}\left(  t\right)  =\frac{du_{n}}{dt}\left(  \beta\left(
t\right)  \right)  \frac{d\beta}{dt}\left(  t\right)  \text{ for a.a. }t>0.
\]
It is clear that $w_{n}\rightarrow w$ in $L^{p}\left(  \varepsilon,T;X\right)
$ and $\dfrac{du_{n}}{dt}\left(  \beta\left(  \text{\textperiodcentered
}\right)  \right)  \rightarrow\dfrac{du}{dt}\left(  \beta\left(
\text{\textperiodcentered}\right)  \right)  $ in $L^{q}\left(  \varepsilon
,T;X^{\ast}\right)  $. Pasing to the limit we obtain that%
\[
\dfrac{dw}{dt}\left(  \text{\textperiodcentered}\right)  =\frac{du}{dt}\left(
\beta\left(  \text{\textperiodcentered}\right)  \right)  \frac{d\beta}%
{dt}\left(  \text{\textperiodcentered}\right)
\]
in the sense of distributions $\mathcal{D}^{\ast}\left(  0,+\infty;X\right)
$. As $\dfrac{du}{dt}\left(  \beta\left(  \text{\textperiodcentered}\right)
\right)  \dfrac{d\beta}{dt}\left(  \text{\textperiodcentered}\right)  \in
L^{q}\left(  \varepsilon,T;X^{\ast}\right)  $, $\dfrac{dw}{dt}\in L^{q}\left(
\varepsilon,T;X^{\ast}\right)  $ and (\ref{EqualityDeriv}) holds true.
\end{proof}

\bigskip

We would like to avoid $a$ being uniformly bounded by above. We can prove the
continuity of $u$ for regular solutions by assuming that $a$ has at most
linear growth.

\begin{lemma}
\label{ContSol}Assume that conditions (\ref{h})-(\ref{3}) hold and also that%
\begin{equation}
a\left(  s\right)  \leq M_{1}+M_{2}s,\ \forall s\geq0, \label{7}%
\end{equation}
for some contants $M_{1},M_{2}\geq0$. Then any regular solution satisfies that
$u\in C([0,T],L^{2}(\Omega))$ for all $T>0.$ Moreover, $w\left(  t\right)
=u\left(  \alpha^{-1}\left(  t\right)  \right)  $, where $\alpha(t)=\int%
_{0}^{t}a(\Vert u(s)\Vert_{H_{0}^{1}}^{2})ds$, is a regular solution to the
problem%
\begin{equation}
\left\{
\begin{array}
[c]{l}%
w_{t}-\Delta w=\dfrac{f(w)+h(t)}{a(\Vert w\Vert_{H_{0}^{1}}^{2})},\ \text{in
}\Omega\times(0,\infty),\\
w=0\quad\text{in }\partial\Omega\times(0,\infty),\\
w(0,x)=u_{0}(x)\quad\text{in }\Omega.
\end{array}
\right.  \label{w}%
\end{equation}

\end{lemma}

\begin{proof}
Condition (\ref{7}) guarantees that $a(\Vert u($\textperiodcentered
$)\Vert_{H_{0}^{1}}^{2})\in L^{1}\left(  0,T\right)  $ if $u\in L^{2}\left(
0,T;H_{0}^{1}\left(  \Omega\right)  \right)  $. We make the following time
rescaling
\[
u(t,x)=w(\alpha(t),x).
\]
As $a(\Vert u($\textperiodcentered$)\Vert_{H_{0}^{1}}^{2})\in L^{1}\left(
0,T\right)  $, the function $t\mapsto\alpha\left(  t\right)  $ is continuous.
It is clear that the function $w\left(  t,x\right)  =u(\alpha^{-1}\left(
t\right)  ,x)$ belongs to the space $L^{\infty}(0,T;L^{2}(\Omega))\cap
L^{2}(0,T;H_{0}^{1}(\Omega))\cap L^{p}(0,T;L^{p}(\Omega))$ and also to the
spaces $L^{\infty}(\varepsilon,T;H_{0}^{1}(\Omega))$ and $L^{2}(\varepsilon
,T;D(A))$ for any $0<\varepsilon<T.$. Moreover, $\dfrac{du}{dt}\in
L^{q}\left(  \varepsilon,T;L^{q}\left(  \Omega\right)  \right)  $ and Lemma
\ref{DerivComposition} give $\dfrac{dw}{dt}\in L^{q}\left(  \varepsilon
,T;L^{q}\left(  \Omega\right)  \right)  $ and%
\begin{equation}
\frac{dw}{dt}\left(  t\right)  =\frac{du}{dt}\left(  \alpha^{-1}\left(
t\right)  \right)  \frac{d}{dt}\alpha^{-1}\left(  t\right)  =\frac{du}%
{dt}\left(  \alpha^{-1}\left(  t\right)  \right)  \frac{1}{a\left(  \Vert
w(t))\Vert_{H_{0}^{1}}^{2}\right)  },\text{ for a.a. }t.
\label{EqualityDeriv2}%
\end{equation}
Equality (\ref{EqualityDeriv}) implies that
\[
\frac{du}{dt}\left(  \alpha^{-1}\left(  t\right)  \right)  -a\left(  \Vert
u(\alpha^{-1}(t))\Vert_{H_{0}^{1}}^{2}\right)  \Delta u\left(  \alpha
^{-1}\left(  t\right)  \right)  =f\left(  u\left(  \alpha^{-1}\left(
t\right)  \right)  \right)  +h(t),\text{ for a.a. }t>0,\text{ }%
\]
so (\ref{EqualityDeriv2}) gives%
\[
\frac{dw}{dt}\left(  t\right)  -\Delta w\left(  t\right)  =\frac
{f(w(t))}{a(\Vert w\left(  t\right)  \Vert_{H_{0}^{1}}^{2})}+\frac
{h(t)}{a(\Vert w\left(  t\right)  \Vert_{H_{0}^{1}}^{2})}\text{ for a.a.
}t>0\text{.}%
\]
Hence, $w$ is a regular solution to problem (\ref{w}). Since $0<\frac{1}%
{a(s)}\leq\frac{1}{m}$, we obtain that
\[
\frac{dw}{dt}\in L^{2}(0,T;H^{-1}(\Omega))+L^{q}(0,T;L^{q}(\Omega)).
\]
Therefore, $w\in C([0,T],L^{2}(\Omega)),$ so that
\[
u\in C([0,T],L^{2}(\Omega)).
\]

\end{proof}

\begin{remark}
\label{Deriv}Under assumptions (\ref{h})-(\ref{3}) any regular solution
$u\left(  \text{\textperiodcentered}\right)  $ satisfies that $\dfrac{du}%
{dt}\in L^{2}\left(  \varepsilon,T;L^{2}\left(  \Omega\right)  \right)
+L^{q}\left(  \varepsilon,T;L^{q}\left(  \Omega\right)  \right)  $ for all
$0<\varepsilon<T$. Then by \cite[p.285]{chepvishik} $u\in C([\varepsilon
,T],L^{2}\left(  \Omega\right)  )$, $t\mapsto\left\Vert u\left(  t\right)
\right\Vert ^{2}$ is absolutely continuous on $[\varepsilon,T]$ and
\[
\frac{d}{dt}\left\Vert u\left(  t\right)  \right\Vert _{L^{2}}^{2}=2\left(
\frac{du}{dt},u\right)  \text{ for a.a. }t>\varepsilon.
\]

\end{remark}

If the initial condition belongs to $H_{0}^{1}\left(  \Omega\right)  \cap
L^{p}\left(  \Omega\right)  $, we can define strong solutions as well.

\begin{definition}
A strong solution to (\ref{1}) is a weak solution with the extra regularity
$u\in L^{\infty}(0,T;H_{0}^{1}(\Omega)\cap L^{p}(\Omega))$, $u\in
L^{2}(0,T;D(A))$ and $\dfrac{du}{dt}\in L^{2}\left(  0,T;L^{2}\left(
\Omega\right)  \right)  $ for any $T>0.$
\end{definition}

We observe that if $u$ is a strong solution, then $u\in C([0,T],H_{0}%
^{1}\left(  \Omega\right)  )$ (see \cite[p.102]{sellyou}). Also, $u\in
L^{\infty}(0,T;L^{p}(\Omega))$ and $u\in C([0,T],L^{2}\left(  \Omega\right)
)$ imply that $u\in C_{w}([0,T],L^{p}(\Omega))$ (see \cite[p.263]{temmam}).
Thus, the initial condition makes sense. Also, the equality $f\left(
u\right)  =u_{t}-a\left(  \left\Vert u\right\Vert _{H_{0}^{1}}^{2}\right)
\Delta u-h$ implies that $f\left(  u\right)  \in L^{2}\left(  0,T;L^{2}\left(
\Omega\right)  \right)  $

Also, if $u$ is a regular solution such that $\dfrac{du}{dt}\in L^{2}\left(
\varepsilon,T;L^{2}\left(  \Omega\right)  \right)  $ for all $0<\varepsilon
<T$, then $u\in C((0,T],H_{0}^{1}\left(  \Omega\right)  ).$

We prove first the existence of regular solutions for initial conditions in
$L^{2}\left(  \Omega\right)  .$

\begin{theorem}
\label{existenceweakregularsolutionu0L2} Let $f\in C^{1}(\mathbb{R})$ be such
that $f^{\prime}(s)\leq\eta$. Also, assume conditions (\ref{h})-(\ref{3}) and
(\ref{7}). Then, for any $u_{0}\in L^{2}(\Omega)$ there exists at least one
regular solution to (\ref{1}).
\end{theorem}

\begin{proof}
We will prove the resuly by using Faedo-Galerkin approximations.

Consider a fixed value $T>0.$ Let be $\{w_{j}\}_{j\geq1}$ the sequence of
eigenfunctions of $-\Delta$ in $H_{0}^{1}(\Omega)$ with homogeneous Dirichlet
boundary conditions, which forms a special basis of $L^{2}(\Omega)$. If
$\Omega$ is a bounded $C^{s}$ domain with $s\geq\max\{n(p-2)/2p,1\}$, it is
well known that $\{w_{j}\}\subset H_{0}^{1}(\Omega)\cap L^{p}(\Omega)$
\cite[Chapter 8]{robinson}. Then for the set $V_{n}=span[w_{1},\ldots,w_{n}]$
we have that $\cup_{n\in\mathbb{N}}V_{n}$ is dense in $L^{2}(\Omega)$ and also
in $H_{0}^{1}(\Omega)\cap L^{p}(\Omega).$ As usual, $P_{n}$ will be the
orthogonal projection in $L^{2}\left(  \Omega\right)  $, that is
\[
u_{n}:=P_{n}u=\sum_{j=1}^{n}(u,w_{j})w_{j}.
\]
For each integer $n\geq1$, we consider the Galerkin approximations
\[
u_{n}(t)=\sum_{j=1}^{n}\gamma_{nj}(t)w_{j},
\]
which satisfy the following nonlinear ODE system
\begin{equation}
\left\{
\begin{array}
[c]{l}%
\dfrac{d}{dt}(u_{n},w_{i})+a(\Vert u_{n}\Vert_{H_{0}^{1}}^{2})(\nabla
u_{n},\nabla w_{i})=(f(u_{n}),w_{i})+(h,w_{i})\quad\forall i=1,\ldots,n,\\
u_{n}(0)=P_{n}u_{0}.
\end{array}
\right.  \label{1.6}%
\end{equation}
where $P_{n}u_{0}\rightarrow u_{0}$ in $L^{2}(\Omega)$, which is equivalent to
the problem
\begin{equation}%
\begin{split}
\frac{du_{n_{j}}}{dt}  &  =-a(\Vert u_{n}\Vert_{H_{0}^{1}}^{2})\lambda
_{j}u_{n_{j}}+(f(u_{n}),w_{j})+(h(t),w_{j}),\\
(u_{n}(0),w_{j})  &  =(u_{0},w_{j}),\quad j=1,\ldots,n,
\end{split}
\label{1.7}%
\end{equation}
where $\lambda_{j}$ is the eigenvalue associated to the eigenfunction $w_{j}$.
Since the right hand side of (\ref{1.7}) is continuous in $u_{n}(t)$ this
Cauchy problem possesses a solution on some interval $[0,t_{n}],0<t_{n}<T$. We
claim that for any $T>0$ such a solution can be extended to the whole interval
$[0,T],$ which follows from a priori estimates in the space $L^{2}(\Omega)$ of
the sequence $\{u_{n}\}$.

Multiplying by $\gamma_{n_{j}}(t)$ and summing from $j=1$ to $n$, we obtain
\begin{equation}
\frac{1}{2}\frac{d}{dt}\Vert u_{n}(t)\Vert_{L^{2}}^{2}+a(\Vert u_{n}%
\Vert_{H_{0}^{1}}^{2})\Vert u_{n}(t)\Vert_{H_{0}^{1}}^{2}=(f(u_{n}%
(t)),u_{n}(t))+(h,u_{n}(t))\ \text{for }\mathit{a.e.}t\in(0,t_{n}).
\label{1.8}%
\end{equation}
Using (\ref{3}) and Young and Poincar\'{e} inequalities we deduce that
\[
(f(u_{n}(t)),u_{n}(t))\leq\kappa|\Omega|-\alpha_{1}\Vert u_{n}(t)\Vert_{L^{p}%
}^{p},
\]
\[
(h(t),u_{n}(t))\leq\frac{m}{2}\Vert u_{n}(t)\Vert_{H_{0}^{1}}^{2}+\frac
{1}{2\lambda_{1}m}\Vert h(t)\Vert_{L^{2}}^{2}.
\]
Hence, from (\ref{1.8}) it follows that
\begin{equation}
\frac{1}{2}\frac{d}{dt}\Vert u_{n}(t)\Vert_{L^{2}}^{2}+\frac{m}{2}\Vert
u_{n}(t)\Vert_{H_{0}^{1}}^{2}+\alpha_{1}\Vert u_{n}(t)\Vert_{L^{p}}^{p}%
\leq\kappa|\Omega|+\frac{1}{2\lambda_{1}m}\Vert h(t)\Vert_{L^{2}}%
^{2}\ \text{for }\mathit{a.e.}t\in(0,t_{n}). \label{1.9}%
\end{equation}
Then, integrating (\ref{1.9}) from $0$ to $t\in(0,t_{n})$ we get
\begin{equation}%
\begin{split}
&  \frac{1}{2}\Vert u_{n}(t)\Vert_{L^{2}}^{2}+\frac{m}{2}\int_{0}^{t}\Vert
u_{n}(s)\Vert_{H_{0}^{1}}^{2}ds+\alpha_{1}\int_{0}^{t}\Vert u_{n}%
(s)\Vert_{L^{p}}^{p}ds\\
&  \leq\kappa|\Omega|t+\frac{1}{2\lambda_{1}m}\int_{0}^{t}\Vert h(s)\Vert
_{L^{2}}^{2}ds+\frac{1}{2}\Vert u_{n}(0)\Vert_{L^{2}}^{2}\leq TK_{2}%
+K_{3}(T)+\frac{1}{2}\Vert u_{n}(0)\Vert_{L^{2}}^{2}.
\end{split}
\label{1.10.1}%
\end{equation}
Therefore, the sequence $\{u_{n}\}$ is well defined and bounded in $L^{\infty
}(0,T;L^{2}(\Omega))\cap L^{2}(0,T;H_{0}^{1}(\Omega))\cap L^{p}(0,T;L^{p}%
(\Omega))$. Also, $\{-\Delta u_{n}\}$ is bounded in $L^{2}(0,T;H^{-1}%
(\Omega)).$

On the other hand, by (\ref{1.4}) it follows that%
\[
\int_{0}^{T}\int_{\Omega}|f(u(x,t))|^{q}dxdt\leq2^{q-1}(C_{1}^{q}%
|\Omega|T+C_{2}^{q}\int_{0}^{T}\Vert u(t)\Vert_{L^{p}}^{p}dt),
\]
with $\frac{1}{p}+\frac{1}{q}=1.$ Hence, since $\{u_{n}\}$ bounded in
$L^{p}(0,T;L^{p}(\Omega))$, $f(u_{n})$ is bounded in $L^{q}(0,T;L^{q}%
(\Omega))$.

On the other hand, multiplying (\ref{1.7}) by $\lambda_{j}\gamma_{nj}(t)$ and
summing from $j=1$ to $n$, we obtain%
\[
\frac{1}{2}\frac{d}{dt}\Vert u_{n}\Vert_{H_{0}^{1}}^{2}+m\Vert\Delta
u_{n}\Vert_{L^{2}}^{2}\leq\langle f(u_{n}),-\Delta u_{n}\rangle+(h(t),-\Delta
u_{n})\leq\eta\Vert u_{n}\Vert_{H_{0}^{1}}^{2}+\frac{1}{2m}\Vert
h(t)\Vert_{L^{2}}^{2}+\frac{m}{2}\Vert\Delta u_{n}\Vert_{L^{2}}^{2}.
\]
Integrating the previous expression between $s$ and $t$, with $0<s\leq t\leq
T,$ we get%
\begin{equation}
\frac{1}{2}\Vert u_{n}(t)\Vert_{H_{0}^{1}}^{2}+\frac{m}{2}\int_{s}^{t}%
\Vert\Delta u_{n}(r)\Vert_{L^{2}}^{2}dr\leq\eta\int_{0}^{T}\Vert u_{n}%
(r)\Vert_{H_{0}^{1}}^{2}dr+\frac{1}{2}\Vert u_{n}(s)\Vert_{H_{0}^{1}}%
^{2}+\frac{1}{2m}\int_{s}^{t}\Vert h(r)\Vert_{L^{2}}^{2}dr. \label{3.2}%
\end{equation}
Now, integrating in $s$ between $0$ and $t$, it follows that%
\[
t\Vert u_{n}(t)\Vert_{H_{0}^{1}}^{2}\leq(2\eta T+1)\int_{0}^{T}\Vert
u_{n}(r)\Vert_{H_{0}^{1}}^{2}dr+K_{3}(T)T.
\]
Hence,
\begin{equation}
\Vert u_{n}(t)\Vert_{H_{0}^{1}}^{2}\leq\frac{2\eta T+1}{\varepsilon}\int%
_{0}^{T}\Vert u_{n}(r)\Vert_{H_{0}^{1}}^{2}dr+\frac{K_{3}(T)T}{\varepsilon}.
\label{absorbing}%
\end{equation}
for all $t\in\lbrack\varepsilon,T]$ with $\varepsilon\in(0,T).$ From the last
inequality and (\ref{1.10.1}) we deduce that $\Vert u_{n}(t)\Vert_{H_{0}^{1}}$
is uniformly bounded in $[\varepsilon,T]$ and by the continuity of the
function $a$ we get that $a(\Vert u_{n}(t)\Vert_{H_{0}^{1}}^{2})$ is bounded
in $[\varepsilon,T]$. Also, it follows that
\begin{equation}
\{u_{n}\}\text{ is bounded in }L^{\infty}(\varepsilon,T;H_{0}^{1}(\Omega)).
\label{coninfi}%
\end{equation}
On the other hand, taking $s=\varepsilon$ and $t=T$ in (\ref{3.2}), by
(\ref{1.10.1}) we obtain that%
\begin{equation}
\{u_{n}\}\text{ is bounded in }L^{2}(\varepsilon,T;D(A)), \label{cotda}%
\end{equation}
so $\{-\Delta u_{n}\}$ and $\{a(\Vert u_{n}\Vert_{H_{0}^{1}}^{2})\Delta
u_{n}\}$ are bounded in $L^{2}(\varepsilon,T;L^{2}(\Omega))$. Thus,
\begin{equation}
\dfrac{du_{n}}{dt}\text{ is bounded in }L^{q}(\varepsilon,T;L^{q}(\Omega)).
\label{CotaDeriv}%
\end{equation}

Therefore, there exists $u\in L^{\infty}(\varepsilon,T;H_{0}^{1}(\Omega))\cap
L^{2}(0,T;H_{0}^{1}(\Omega))\cap L^{\infty}(0,T;L^{2}(\Omega))$ and a
subsequence $u_{n}$, relabelled the same, such that%
\begin{equation}%
\begin{split}
u_{n}  &  \overset{\ast}{\rightharpoonup}u\text{ in }L^{\infty}(\varepsilon
,T;H_{0}^{1}(\Omega)),\\
u_{n}  &  \overset{\ast}{\rightharpoonup}u\text{ in }L^{\infty}(0,T;L^{2}%
(\Omega)),\\
u_{n}  &  \rightharpoonup u\text{ in }L^{2}(0,T;H_{0}^{1}(\Omega)),\\
u_{n}  &  \rightharpoonup u\text{ in }L^{p}(0,T;L^{p}(\Omega)),\\
u_{n}  &  \rightharpoonup u\text{ in }L^{2}(\varepsilon,T;D(A)),\\
\frac{du_{n}}{dt}  &  \rightharpoonup\frac{du}{dt}\text{ in }L^{q}%
(\varepsilon,T;L^{q}(\Omega)),\\
f(u_{n})  &  \rightharpoonup\chi\text{ in }L^{q}(0,T;L^{q}(\Omega)),\\
{a(\Vert u_{n}\Vert_{H_{0}^{1}}^{2})}  &  \overset{\ast}{\rightharpoonup
}b\text{ in }L^{\infty}(\varepsilon,T),
\end{split}
\label{1.29}%
\end{equation}
for any $0<\varepsilon<T$, where $\rightharpoonup$ means weak convergence.

Moreover, by (\ref{cotda})-(\ref{CotaDeriv}) the Aubin-Lions Compactness Lemma
gives that $u_{n}\rightarrow u$ in $L^{2}(\varepsilon,T;H_{0}^{1}(\Omega))$,
so $u_{n}(t)\rightarrow u(t)$ a.e. on $(\varepsilon,T)$ for any $\varepsilon
>0$. Consequently, there exists a subsequence $u_{n}$, relabelled the same,
such that $u_{n}\left(  t,x\right)  \rightarrow u\left(  t,x\right)  $ a.e. in
$\Omega\times(0,T)$. Also, we know that $P_{n}f(u_{n})\rightharpoonup\chi$
(see \cite[p.224]{robinson}). Since $f$ is continuous, it follows that
$f(u_{n}\left(  t,x\right)  )\rightarrow f(u\left(  t,x\right)  )$ a.e. in
$\Omega\times(0,T)$. Therefore, in view of (\ref{1.29}), by \cite[Lemma
1.3]{lions} we have that $\chi=f(u)$.

As a consequence, by the continuity of $a$, we get that
\[
a(\Vert u_{n}(t)\Vert_{H_{0}^{1}}^{2})\rightarrow a(\Vert u(t)\Vert_{H_{0}%
^{1}}^{2})\quad\text{ a.e. on }(\varepsilon,T).
\]
Since the sequence is bounded, by Lebesgue's theorem this convergence takes
place in $L^{2}(\varepsilon,T)$ and $b=a(\Vert u\Vert_{H_{0}^{1}}^{2})$ on
$(\varepsilon,T)$. Thus,
\begin{equation}
a(\Vert u_{n}\Vert_{H_{0}^{1}}^{2})\Delta u_{n}\rightharpoonup a(\Vert
u\Vert_{H_{0}^{1}}^{2})\Delta u,\quad\text{ in }L^{2}(\varepsilon
,T;L^{2}(\Omega)). \label{convergencianolocal}%
\end{equation}

Finally, since $\{w_{i}\}$ is dense in $H_{0}^{1}(\Omega)\cap L^{p}(\Omega)$,
in view of (\ref{1.29}) and (\ref{convergencianolocal}), we can pass to the
limit in (\ref{1.6}) and conclude that (\ref{equationweak}) holds for all
$v\in H_{0}^{1}(\Omega)\cap L^{p}(\Omega)$.

To conclude the proof, we have to check that $u(0)=u_{0}$. Indeed, let be
$\phi\in C^{1}([0,T]);H_{0}^{1}(\Omega)\cap L^{p}(\Omega))$, with $\phi(T)=0$,
$\phi(0)\not =0$. We consider the functions $w\left(  t\right)  =u(\alpha
^{-1}\left(  t\right)  ),\ w_{n}\left(  t\right)  =u_{n}\left(  \alpha
^{-1}\left(  t\right)  \right)  $, which are regular solutions to problem
(\ref{w}) with initial conditions $w\left(  0\right)  =u_{0},\ w_{n}\left(
0\right)  =u_{n}\left(  0\right)  =P_{n}u_{0}$. Since $\dfrac{dw}{dt}\in
L^{2}(0,T;H^{-1}(\Omega))+L^{q}(0,T;L^{q}(\Omega))$, we can multiply the
equation in (\ref{w}) by $\phi$ and integrate by parts in the $t$ variable to
obtain that%
\begin{equation}
\int_{0}^{T}\left(  -\left(  w\left(  t\right)  ,\phi^{\prime}\left(
t\right)  \right)  -\left(  \Delta w\left(  t\right)  ,\phi\left(  t\right)
\right)  \right)  dt=\int_{0}^{T}\left(  \frac{f(w(t))+h(t)}{a(\Vert
w(t)\Vert_{H_{0}^{1}}^{2})},\phi\left(  t\right)  \right)  dt+\left(  w\left(
0\right)  ,\phi\left(  0\right)  \right)  , \label{In1}%
\end{equation}%
\begin{equation}
\int_{0}^{T}\left(  -\left(  w_{n}\left(  t\right)  ,\phi^{\prime}\left(
t\right)  \right)  -\left(  \Delta w_{n}\left(  t\right)  ,\phi\left(
t\right)  \right)  \right)  dt=\int_{0}^{T}\left(  \frac{f(w_{n}%
(t))+h(t)}{a(\Vert w_{n}(t)\Vert_{H_{0}^{1}}^{2})},\phi\left(  t\right)
\right)  dt+\left(  w_{n}\left(  0\right)  ,\phi\left(  0\right)  \right)  .
\label{In2}%
\end{equation}
We can easily obtain by the previous convergences that%
\begin{align*}
w_{n}  &  \rightharpoonup w\text{ in }L^{2}\left(  0,T;H_{0}^{1}\left(
\Omega\right)  \right)  ,\\
\Delta w_{n}  &  \rightharpoonup\Delta w\text{ in }L^{2}\left(  0,T;H^{-1}%
\left(  \Omega\right)  \right)  ,\\
\frac{f(w_{n}(t))+h\left(  t\right)  }{a(\Vert w_{n}(t)\Vert_{H_{0}^{1}}%
^{2})}  &  \rightharpoonup\frac{f(w(t))+h\left(  t\right)  }{a(\Vert
w(t)\Vert_{H_{0}^{1}}^{2})}\text{ in }L^{q}\left(  0,T;L^{q}\left(
\Omega\right)  \right)  .
\end{align*}
Passing to the limit in (\ref{In1}), taking in to accont (\ref{In2}) and
bearing in mind $w_{n}(0)=P_{n}u_{0}\rightarrow u_{0}$ we get%
\[
\left(  w\left(  0\right)  ,\phi\left(  0\right)  \right)  =\left(
u_{^{\prime}},\phi\left(  0\right)  \right)  .
\]
Since $\phi\left(  0\right)  \in H_{0}^{1}(\Omega)\cap L^{p}(\Omega)$ is
arbitrary, we infer that $w(0)=u\left(  0\right)  =u_{0}$.

Hence, $u$ is a regular solution to (\ref{1}) satisfying $u\left(  0\right)
=u_{0}.$
\end{proof}

\bigskip

Second, we will prove the existence of strong solutions for initial conditions
in $H_{0}^{1}(\Omega)\cap L^{p}(\Omega).$ Im this case we do not need to
impose the upper bound (\ref{7}) of the function $a.$

\begin{theorem}
\label{existencestrongsolutionu0H01Lp} Let $f\in C^{1}(\mathbb{R})$ be such
that $f^{\prime}(s)\leq\eta$. Suppose that conditions (\ref{h})-(\ref{3}) are
fulfilled. Then, for any $u_{0}\in H_{0}^{1}(\Omega)\cap L^{p}(\Omega)$ there
exists at least a strong solution to (\ref{1}).
\end{theorem}

\begin{proof}
We consider, as in Theorem \ref{existenceweakregularsolutionu0L2}, the
Galerkin approximations $u_{n},$ for which (\ref{1.29}) holds. Under the
aforementioned conditions, we will obtain that $u_{n}$ converges to a strong
solution to (\ref{1}). In this proof it is important to observe that
$P_{n}u_{0}\rightarrow u_{0}$ in the spaces $H_{0}^{1}\left(  \Omega\right)  $
and $L^{p}\left(  \Omega\right)  $ \cite[p.199 and 220]{robinson}. Thus, the
sequences $\left\Vert P_{n}u_{0}\right\Vert _{H_{0}^{1}}$ and $\left\Vert
P_{n}u_{0}\right\Vert _{L^{p}}$ are bounded.

First, we the equation in (\ref{1}) by $\dfrac{du_{n}}{dt}$ to obtain%
\[
\Vert\frac{d}{dt}u_{n}(t)\Vert_{L^{2}}^{2}+a(\Vert u_{n}\Vert_{H_{0}^{1}}%
^{2})\frac{1}{2}\frac{d}{dt}\Vert u_{n}\Vert_{H_{0}^{1}}^{2}=\frac{d}{dt}%
\int_{\Omega}\mathcal{F}(u_{n})dx+(h(t),\frac{du_{n}}{dt}).
\]
Introducing
\begin{equation}
A(s)=\int_{0}^{s}a(r)dr, \label{funcionA}%
\end{equation}
we have%
\begin{equation}
\frac{1}{2}\Vert\frac{d}{dt}u_{n}(t)\Vert_{L^{2}}^{2}+\frac{d}{dt}\left[
\frac{1}{2}A(\Vert u_{n}\Vert_{H_{0}^{1}}^{2})-\int_{\Omega}\mathcal{F}%
(u_{n}(t))dx\right]  \leq\frac{1}{2}\Vert h(t)\Vert_{L^{2}}^{2}.
\label{1.1118}%
\end{equation}

Now, integrating (\ref{1.1118}) we have
\[
\frac{1}{2}\int_{0}^{t}\Vert\frac{d}{ds}u_{n}(s)\Vert_{L^{2}}^{2}ds+\frac
{1}{2}A(\Vert u_{n}(t)\Vert_{H_{0}^{1}}^{2})-\int_{\Omega}\mathcal{F}%
(u_{n}(t))dx
\]%
\[
\leq\frac{1}{2}A(\Vert u_{n}(0)\Vert_{H_{0}^{1}}^{2})-\int_{\Omega}%
\mathcal{F}(u_{n}(0))dx+\frac{1}{2}\int_{0}^{t}\Vert h(s)\Vert_{L^{2}}^{2}ds.
\]
From (\ref{2}) and (\ref{4}) we get
\begin{equation}%
\begin{split}
&  \frac{m}{2}\Vert u_{n}(t)\Vert_{H_{0}^{1}}^{2}+\widetilde{\alpha}_{1}\Vert
u_{n}(t)\Vert_{L^{p}}^{p}+\frac{1}{2}\int_{0}^{t}\Vert\frac{d}{ds}%
u_{n}(s)\Vert_{L^{2}}^{2}ds\\
&  \leq\frac{1}{2}A(\Vert u_{n}(0)\Vert_{H_{0}^{1}}^{2})+\widetilde{\alpha
}_{2}\Vert u_{n}(0)\Vert_{L^{p}}^{p}+K.
\end{split}
\label{4.55}%
\end{equation}
Now, from (\ref{4.55}) we obtain that
\begin{equation}
\frac{du_{n}}{dt}\text{ is bounded in }L^{2}(0,T;L^{2}(\Omega)), \label{4.55b}%
\end{equation}
so
\begin{equation}
\frac{du_{n}}{dt}\rightharpoonup\frac{du}{dt}\text{ in }L^{2}(0,T;L^{2}%
(\Omega)). \label{4.56}%
\end{equation}
On the other hand, the embeddings $H_{0}^{1}(\Omega)\subset\subset
L^{2}(\Omega)\subset H^{-1}(\Omega)$ (where $\subset\subset$ means a compact
embedding) and Aubin-Lion's Compactness Lemma imply that
\[
u_{n}\rightarrow u\text{ in }L^{2}(0,T;L^{2}(\Omega)).
\]
Hence,
\[
u_{n}\rightarrow u\text{ for a.a. }(x,t)\in\Omega\times(0,T).
\]
Moreover, thanks to the following inequality
\[
|u_{n}(t_{2})-u_{n}(t_{1})|^{2}=\left\vert \int_{t_{1}}^{t_{2}}\frac{d}%
{dt}u_{n}(s)ds\right\vert ^{2}\leq\Vert\frac{d}{dt}u_{n}\Vert_{L^{2}%
(0,T;L^{2}(\Omega))}^{2}\ |t_{2}-t_{1}|\quad\forall t_{1},t_{2}\in
\lbrack0,T],
\]
(\ref{4.55b}) and $H_{0}^{1}(\Omega)\subset\subset L^{2}(\Omega)$, \ the
Ascoli-Arzel\'{a}'s theorem implies that $\{u_{n}\}$ converges strongly in
$C([0,T];L^{2}(\Omega))$ for all $T>0$. Therefore, we obtain from (\ref{4.55})
that $u_{n}(t)\rightharpoonup u(t)\text{ in }H_{0}^{1}(\Omega)\cap
L^{p}(\Omega)$, for any $t\geq0$, and
\begin{equation}
u_{n}\overset{\ast}{\rightharpoonup}u\text{ in }L^{\infty}(0,T;H_{0}%
^{1}(\Omega)\cap L^{p}(\Omega)). \label{5.88}%
\end{equation}
Also, by the continuity of the function $a$, $a(\Vert u_{n}\left(  t\right)
\Vert_{H_{0}^{1}}^{2})$ is uniformly bounded in $[0,T]$.

Multiplying (\ref{1.8}) by $\lambda_{j}\gamma_{nj}(t)$ and summing from $j=1$
to $n$, we obtain
\[
\frac{1}{2}\frac{d}{dt}\Vert u_{n}\Vert_{H_{0}^{1}}^{2}+m\Vert-\Delta
u_{n}\Vert_{L^{2}}^{2}=(f(u_{n}),-\Delta u_{n})+(h(t),-\Delta u)\leq\eta\Vert
u_{n}\Vert_{H_{0}^{1}}^{2}+\frac{1}{2m}\Vert h(t)\Vert_{L^{2}}^{2}+\frac{m}%
{2}\Vert-\Delta u_{n}\Vert_{L^{2}}^{2}.
\]
Integrating the previous expression between $0$ and $T$ it follows that%
\begin{equation}
\frac{1}{2}\Vert u_{n}(T)\Vert_{H_{0}^{1}}^{2}+\frac{m}{2}\int_{0}^{T}%
\Vert-\Delta u_{n}(s)\Vert_{L^{2}}^{2}ds\leq\eta\int_{0}^{T}\Vert
u_{n}(t)\Vert_{H_{0}^{1}}^{2}dt+\frac{1}{2}\Vert u_{n}(0)\Vert_{H_{0}^{1}}%
^{2}+K. \label{ultima223}%
\end{equation}
Finally, taking into account (\ref{1.10.1}), from (\ref{ultima223}) we get
\[
u_{n}\text{ is uniformly bounded in }L^{2}(0,T;D(A)),
\]
so
\begin{equation}
u_{n}\rightharpoonup u\text{ in }L^{2}(0,T;D(A)). \label{4.111}%
\end{equation}
Arguing as in Theorem \ref{existenceweakregularsolutionu0L2} we also obtain
that
\begin{align}
u_{n}  &  \rightarrow u\text{ in }L^{2}\left(  0,T;H_{0}^{1}\left(
\Omega\right)  \right)  ,\nonumber\\
a\left(  \Vert u_{n}\Vert_{H_{0}^{1}}^{2}\right)   &  \rightarrow a\left(
\Vert u\Vert_{H_{0}^{1}}^{2}\right)  \text{ in }L^{1}\left(  0,T\right)
,\nonumber\\
f\left(  u_{n}\right)   &  \rightharpoonup f\left(  u\right)  \text{ in }%
L^{q}\left(  0,T;L^{q}\left(  \Omega\right)  \right)  ,\nonumber\\
a\left(  \Vert u_{n}\Vert_{H_{0}^{1}}^{2}\right)  u_{n}  &  \rightharpoonup
a\left(  \Vert u\Vert_{H_{0}^{1}}^{2}\right)  u\text{ in }L^{2}(0,T;D(A)).
\label{Converg}%
\end{align}

Therefore, we can pass to the limit to conclude that $u$ is a strong solution.

It remains to show that $u\left(  0\right)  =u_{0}$. This is done, in a
similar way as in Theorem \ref{existenceweakregularsolutionu0L2}, by
multiplying the equation in (\ref{1}) by a function $\phi\in C^{1}%
([0,T]);H_{0}^{1}(\Omega)\cap L^{p}(\Omega))$, with $\phi(T)=0$,
$\phi(0)\not =0$ for the Galerkin approximations $u_{n}$ and the limit
function $u$ and integrating by parts. Then taking into account the above
convergences and $P_{n}u_{0}\rightarrow u_{0}$ in $L^{2}\left(  \Omega\right)
$ we obtain that $u\left(  0\right)  =u_{0}.$
\end{proof}

\bigskip

We can avoid using condition $f^{\prime}\leq\eta$ by imposing extra
assumptions on the constant $p$. Indeed, if
\begin{equation}
p\leq\frac{2n-2}{n-2},\text{ for }n\geq3, \label{Condp}%
\end{equation}
($p\geq2$ is arbitrary for $n=1,2,$), then the embdedding $H_{0}^{1}\left(
\Omega\right)  \subset L^{2\left(  p-1\right)  }\left(  \Omega\right)  \subset
L^{p}\left(  \Omega\right)  $ and (\ref{1.4}) imply that%
\begin{equation}
||f(u(t))||_{L^{2}}^{2}\leq2C(1+\int_{\Omega}|u(t,x)|^{2(p-1)}dx)\leq
\widetilde{C}\left(  1+\left\Vert u\left(  t\right)  \right\Vert _{H_{0}^{1}%
}^{2\left(  p-1\right)  }\right)  , \label{Acotacionfu}%
\end{equation}
so
\begin{equation}
f(u)\in L^{2}(0,T;L^{2}(\Omega))\quad\text{ } \label{valordepParau0enH1}%
\end{equation}
provided that $u\in L^{\infty}(0,T;H_{0}^{1}(\Omega))$. Morever, $f\left(
A\right)  $ is bounded in $L^{2}(0,T;L^{2}(\Omega)$ if $A$ is a bounded set of
$L^{\infty}(0,T;H_{0}^{1}(\Omega))$.

\begin{theorem}
\label{existenciau0paper2014} Assume that (\ref{h})-(\ref{3}) and
(\ref{Condp}) hold. Then for any $u_{0}\in H_{0}^{1}(\Omega)$ there exists at
least one strong solution to (\ref{1}).
\end{theorem}

\begin{proof}
Reasoning as in Theorem \ref{existencestrongsolutionu0H01Lp} and considering
as well the Galerkin scheme, (\ref{1.29}), (\ref{4.56}) and (\ref{5.88}) hold.
We just need to check that (\ref{4.111}) is also true, as after that the proof
is finished repeating the same lines of Theorem
\ref{existencestrongsolutionu0H01Lp}.

Multiplying (\ref{1.8}) by $\lambda_{j}\gamma_{nj}(t)$ and summing from $j=1$
to $n$, we obtain
\[%
\begin{split}
\frac{1}{2}\frac{d}{dt}\Vert u_{n}\Vert_{H_{0}^{1}}^{2}+m\Vert-\Delta
u_{n}\Vert_{L^{2}}^{2}  &  =(f(u_{n}),-\Delta u_{n})+(h(t),-\Delta u)\\
&  \leq\frac{1}{2m}\Vert f(u_{n})\Vert_{L^{2}}^{2}+\frac{m}{2}\Vert-\Delta
u_{n}\Vert_{L^{2}}^{2}+\frac{1}{m}\Vert h(t)\Vert_{L^{2}}^{2}+\frac{m}{4}%
\Vert-\Delta u_{n}\Vert_{L^{2}}^{2}.
\end{split}
\]
Integrating the previous expression between $0$ and $T$ it follows that%
\begin{equation}%
\begin{split}
&  \frac{1}{2}\Vert u_{n}(T)\Vert_{H_{0}^{1}}^{2}+\frac{m}{4}\int_{0}^{T}%
\Vert-\Delta u_{n}(s)\Vert_{L^{2}}^{2}ds\\
&  \leq\frac{1}{m}\int_{0}^{T}\Vert f(u_{n}(t))\Vert_{L^{2}}^{2}dt+\frac{1}%
{2}\Vert u_{n}(0)\Vert_{H_{0}^{1}}^{2}+\frac{1}{m}\int_{0}^{T}\Vert
h(t)\Vert_{L^{2}}^{2}dt.
\end{split}
\label{ultima2}%
\end{equation}
In view of (\ref{5.88}) and (\ref{Acotacionfu}), we have that $f\left(
u\right)  $ is bounded in $L^{2}\left(  0,T;L^{2}\left(  \Omega\right)
\right)  $, so from (\ref{ultima2}) we get that $u_{n}$ is bounded in
$L^{2}(0,T;D(A)).$ Therefore,
\begin{equation}
u_{n}\rightharpoonup u\text{ in }L^{2}(0,T;D(A)), \label{5.101}%
\end{equation}
as required.
\end{proof}

\bigskip

In the case of regular solutions we can get rid of the condition $f^{\prime
}\leq\eta$ as well by imposing the extra assumption (\ref{Condp}) on the
constant $p$.

\begin{theorem}
\label{existenceweaksolutionu0L2sinderivadaversion2014} Assume that
(\ref{h})-(\ref{3}), (\ref{7}) and (\ref{Condp}) hold. Then, for any $u_{0}\in
L^{2}(\Omega)$ there exists at least one regular solution to (\ref{1}).
\end{theorem}

\begin{proof}
Let $u_{0}^{n}\in H_{0}^{1}(\Omega)$ be a sequence such that $u_{0}%
^{n}\rightarrow u_{0}$ in $L^{2}(\Omega).$ By Theorem
\ref{existenciau0paper2014} there exists a strong solution $u^{n}(\cdot)$ of
(\ref{1}) with $u^{n}(0)=u_{0}^{n}$. Since $u^{n}\in L^{2}\left(  0,T;D\left(
A\right)  \right)  $ and $\dfrac{du^{n}}{dt}\in L^{2}\left(  0,T;L^{2}\left(
\Omega\right)  \right)  $, from \cite[p.102]{sellyou} the equality
\[
\frac{d}{dt}\Vert u^{n}\Vert_{H_{0}^{1}}^{2}=2(-\Delta u^{n},u_{t}^{n})
\]
holds true for a.a. $t>0$.

Now, multiplying (\ref{1}) by $u^{n}$ and using (\ref{3}) it follows that
\begin{align}
&  \frac{1}{2}\frac{d}{dt}\Vert u^{n}(t)\Vert_{L^{2}}^{2}+m\Vert u^{n}%
\Vert_{H_{0}^{1}}^{2}+\alpha_{1}\Vert u^{n}(t)\Vert_{L^{p}}^{p}\label{44}\\
&  \leq\kappa|\Omega|+\Vert h(t)\Vert_{L^{2}}\Vert u^{n}(t)\Vert_{L^{2}}%
\leq\kappa|\Omega|+\frac{1}{2m\lambda_{1}}\Vert h(t)\Vert_{L^{2}}^{2}-\frac
{m}{2}\Vert u^{n}(t)\Vert_{H_{0}^{1}}^{2},\nonumber
\end{align}
so
\begin{equation}
\Vert u^{n}(t)\Vert_{L^{2}}^{2}\leq\Vert u^{n}(0)\Vert_{L^{2}}^{2}+K_{1}(T).
\label{45}%
\end{equation}
Thus, integrating in (\ref{44}) between $t$ and $t+r$ we get
\begin{equation}%
\begin{split}
&  \Vert u^{n}(t+r)\Vert_{L^{2}}^{2}+m\int_{t}^{t+r}\Vert u^{n}(s)\Vert
_{H_{0}^{1}}^{2}ds+2\alpha_{1}\int_{t}^{t+r}\Vert u^{n}(s)\Vert_{L^{p}}%
^{p}ds\\
&  \leq2\kappa|\Omega|r+\frac{1}{m\lambda_{1}}\int_{t}^{t+r}\Vert
h(s)\Vert_{L^{2}}^{2}ds+\Vert u^{n}(t)\Vert_{L^{2}}^{2}\leq\Vert u^{n}%
(0)\Vert_{L^{2}}^{2}+K_{2}(T).
\end{split}
\label{46}%
\end{equation}
Also, by (\ref{4}) and (\ref{7}) we deduce that
\begin{equation}%
\begin{split}
&  \int_{t}^{t+r}\left(  \frac{1}{2}A(\Vert u^{n}(s)\Vert_{H_{0}^{1}}%
^{2})-\int_{\Omega}\mathcal{F}(u^{n}(s))dx\right)  ds\\
&  \leq\int_{t}^{t+r}\left(  M_{1}\Vert u^{n}\left(  s\right)  \Vert
_{H_{0}^{1}}+\frac{M_{2}}{2}\Vert u^{n}(s)\Vert_{H_{0}^{1}}^{2}\right)
ds+\widetilde{\kappa}\left\vert \Omega\right\vert r+\widetilde{\alpha}_{2}%
\int_{t}^{t+r}\Vert u^{n}(s)\Vert_{L^{p}}^{p}ds\\
&  \leq K_{3}(T)\left(  1+\Vert u^{n}(0)\Vert_{L^{2}}^{2}\right)  ,
\end{split}
\label{47}%
\end{equation}
for all $n>0$ and $t\geq0$.

On the other hand, multiplying (\ref{1}) by $u_{t}^{n}$ we have
\begin{equation}
\frac{1}{2}\Vert u_{t}^{n}(t)\Vert_{L^{2}}^{2}+\frac{d}{dt}\left(  \frac{1}%
{2}A(\Vert u^{n}(t)\Vert_{H_{0}^{1}}^{2})-\int_{\Omega}\mathcal{F}%
(u^{n}(t))dx\right)  =\frac{1}{2}\Vert h(t)\Vert_{L^{2}}^{2}. \label{43}%
\end{equation}
By the uniform Gronwall lemma \cite{Temam88} we obtain
\begin{equation}
\frac{1}{2}A(\Vert u^{n}(t+r)\Vert_{H_{0}^{1}}^{2})-\int_{\Omega}%
\mathcal{F}(u^{n}(t+r))dx\leq\frac{K_{3}(T)(1+\Vert u^{n}(0)\Vert_{L^{2}}%
^{2})}{r}+K_{4}(T),\quad\text{ for all }0\leq t\leq t+r, \label{49}%
\end{equation}
so that by (\ref{2}) and (\ref{4}) we obtain that
\begin{equation}
\Vert u^{n}(t+r)\Vert_{H_{0}^{1}}^{2}+\left\Vert u^{n}\left(  t+r\right)
\right\Vert _{L^{p}}^{p}\leq\frac{K_{5}(T)(1+\Vert u^{n}(0)\Vert_{L^{2}}^{2}%
)}{r}+K_{6}(T), \label{50}%
\end{equation}
for all $t\geq0$. Therefore, the sequence $u^{n}(\cdot)$ is bounded in
$L^{\infty}(r,T;H_{0}^{1}(\Omega))$ for all $0<r<T$. Consequently, $a(\Vert
u^{n}\left(  \text{\textperiodcentered}\right)  \Vert_{H_{0}^{1}}^{2})$ is
bounded in $[r,T]$.

Integrating (\ref{43}) over $(r,T)$, from (\ref{2}), (\ref{4}) and (\ref{49})
it follows that
\begin{equation}%
\begin{split}
&  \frac{1}{2}\int_{r}^{T}\Vert\frac{d}{dt}u^{n}(t)\Vert_{L^{2}}^{2}%
dt+\frac{m}{2}\Vert u^{n}(T)\Vert_{H_{0}^{1}}^{2}+\widetilde{\alpha}_{1}\Vert
u^{n}(T)\Vert_{L^{p}}^{p}-\kappa|\Omega|\\
&  \leq\frac{1}{2}\int_{r}^{T}\Vert\frac{d}{dt}u^{n}(t)\Vert_{L^{2}}%
^{2}dt+\frac{1}{2}A(\Vert u^{n}(T)\Vert_{H_{0}^{1}}^{2})-\int_{\Omega
}\mathcal{F}(u^{n}(T))dx\\
&  \leq\frac{1}{2}\int_{r}^{T}\Vert h(t)\Vert_{L^{2}}^{2}dt+\frac{1}{2}A(\Vert
u^{n}(r)\Vert_{H_{0}^{1}}^{2})-\int_{\Omega}\mathcal{F}(u^{n}(r))dx\\
&  \leq\frac{1}{2}\int_{r}^{T}\Vert h(t)\Vert_{L^{2}}^{2}dt+\frac
{K_{3}(T)(1+\Vert u^{n}(0)\Vert_{L^{2}}^{2})}{r}+K_{4}(T).
\end{split}
\label{51}%
\end{equation}
Thus $\dfrac{du^{n}}{dt}$ is bounded in $L^{2}(r,T;L^{2}(\Omega))$ for all
$0<r<T$.

Taking into account (\ref{Acotacionfu}) and (\ref{50}) we infer that $f\left(
u^{n}\right)  $ is bounded in $L^{2}\left(  r,T;L^{2}\left(  \Omega\right)
\right)  $. By this way, the equality $a(\Vert u^{n}\Vert_{H_{0}^{1}}%
^{2})\Delta u^{n}=u_{t}^{n}-f(u^{n})+h(t)$ implies that $u^{n}$ and $a(\Vert
u^{n}\Vert_{H_{0}^{1}}^{2})\Delta u^{n}$ are bounded in $L^{2}(r,T;D(A))$ and
$L^{2}(r,T;L^{2}(\Omega))$, respectively, for all $0<r<T$.

By the compact embedding $H_{0}^{1}(\Omega)\subset L^{2}(\Omega)$, we can
apply the Ascoli-Arzel\`{a} theorem and obtain that, up to a sequence, there
exists a function $u$ such that
\begin{equation}%
\begin{split}
u^{n}\overset{\ast}{\rightharpoonup}u  &  \text{ in }L^{\infty}(r,T;H_{0}%
^{1}(\Omega)),\\
u^{n}\rightarrow u  &  \text{ in }C([r,T],L^{2}(\Omega)),\\
u^{n}\rightharpoonup u  &  \text{ in }L^{2}(r,T;D(A)),\\
\frac{du^{n}}{dt}\rightharpoonup\frac{du}{dt}  &  \text{ in }L^{2}%
(r,T;L^{2}(\Omega)),
\end{split}
\label{53}%
\end{equation}
for all $0<r<T$.

On the other hand, from (\ref{46}) we infer that $u^{n}$ is bounded in
$L^{\infty}(0,T;L^{2}(\Omega))\cap L^{2}(0,T;H_{0}^{1}(\Omega))\cap
L^{p}(0,T;L^{p}(\Omega))$, for all $T>0$. Therefore, there exists a
subsequence $u^{n}$, relabelled the same, such that
\begin{equation}%
\begin{split}
u^{n}\overset{\ast}{\rightharpoonup}u  &  \text{ in }L^{\infty}(0,T;L^{2}%
(\Omega)),\\
u^{n}\rightharpoonup u  &  \text{ in }L^{2}(0,T;H_{0}^{1}(\Omega)),\\
u^{n}\rightharpoonup u  &  \text{ in }L^{p}(0,T;L^{p}(\Omega)),
\end{split}
\label{55}%
\end{equation}
for all $T>0$. On the other hand, arguing as in the proof of Theorem
\ref{existenceweakregularsolutionu0L2} we obtain that
\begin{align*}
f(u^{n})  &  \rightharpoonup f(u)\text{ in }L^{q}(0,T;L^{q}(\Omega)),\\
u^{n}  &  \rightarrow u\text{ in }L^{2}(r,T;H_{0}^{1}(\Omega)),\\
a(\Vert u^{n}\Vert_{H_{0}^{1}}^{2})  &  \rightarrow a(\Vert u\Vert_{H_{0}^{1}%
}^{2})\text{ in }L^{1}\left(  0,T\right)  ,\\
a(\Vert u^{n}(t)\Vert_{H_{0}^{1}}^{2})\Delta u^{n}  &  \rightharpoonup a(\Vert
u(t)\Vert_{H_{0}^{1}}^{2})\Delta u\quad\text{in }L^{2}(r,T;L^{2}(\Omega)).
\end{align*}

Passing to the limit we obtain that $u\left(  \text{\textperiodcentered
}\right)  $ is a regular solution.

Finally, by a similar argument as in the proof of Theorem
\ref{existenceweakregularsolutionu0L2} we establish that $u\left(  0\right)
=u_{0}.$
\end{proof}

\begin{remark}
Under the conditions of Theorem
\ref{existenceweaksolutionu0L2sinderivadaversion2014} any regular solution
$u\left(  \text{\textperiodcentered}\right)  $ satisfies from
(\ref{Acotacionfu}) that $f\left(  u\right)  \in L^{2}\left(  [\varepsilon
,T];L^{2}\left(  \Omega\right)  \right)  $ for all $0<\varepsilon<T$, and then
$\dfrac{du}{dt}\in L^{2}\left(  [\varepsilon,T];L^{2}\left(  \Omega\right)
\right)  $ as well. Hence, $u\in C((0,T],H_{0}^{1}\left(  \Omega\right)  )$
for all $T>0.$
\end{remark}

\bigskip

We finish this section by giving a sufficient condition ensuring the
uniqueness of solutions.

\begin{theorem}
\label{uniqueness} Assume the conditions of Theorem
\ref{existenceweakregularsolutionu0L2} and additionally that the function
\begin{equation}
s\mapsto a(s^{2})s \label{1.36}%
\end{equation}
is nondecreasing. Then there can exists at most one regular solution to the
Cauchy problem (\ref{1}) for $u_{0}\in L^{2}\left(  \Omega\right)  .$

If, moreover, $M_{2}=0$ in condition (\ref{7}), then there can be at most one
weak solution.

Under the conditions of Theorem \ref{existencestrongsolutionu0H01Lp}, there
can exists at most one strong solution to the Cauchy problem (\ref{1}) for
$u_{0}\in H_{0}^{1}\left(  \Omega\right)  \cap L^{p}\left(  \Omega\right)  .$
\end{theorem}

\begin{proof}
Suppose that $u$ and $v$ are two regular solutions to (\ref{1}) with the same
initial condition $u_{0}=v_{0}$, then by subtractionand multiplying by $u-v$
we get by Remark \ref{Deriv} that
\[
\frac{1}{2}\frac{d}{dt}\Vert u-v\Vert_{L^{2}}^{2}+\langle-a(\Vert u\left(
t\right)  \Vert_{H_{0}^{1}}^{2})\Delta u+a(\Vert v\left(  t\right)
\Vert_{H_{0}^{1}}^{2})\Delta v,u-v\rangle=(f(u)-f(v),u-v).
\]
Let us consider%
\[
I=\langle-a(\Vert u\left(  t\right)  \Vert_{H_{0}^{1}}^{2})\Delta u+a(\Vert
v\left(  t\right)  \Vert_{H_{0}^{1}}^{2})\Delta v,u-v\rangle.
\]
After integrating by parts, we obtain
\begin{align}
I  &  =\int_{\Omega}(a(\Vert u\left(  t\right)  \Vert_{H_{0}^{1}}^{2})|\nabla
u|^{2}-a(\Vert u\left(  t\right)  \Vert_{H_{0}^{1}}^{2})\nabla u\nabla
v-a(\Vert v\left(  t\right)  \Vert_{H_{0}^{1}}^{2})\nabla u\nabla v+a(\Vert
v\left(  t\right)  \Vert_{H_{0}^{1}}^{2})|\nabla v|^{2})dx\nonumber\\
&  \geq a(\Vert u\left(  t\right)  \Vert_{H_{0}^{1}}^{2})\Vert u\left(
t\right)  \Vert_{H_{0}^{1}}^{2}-\left(  a(\Vert u\left(  t\right)
\Vert_{H_{0}^{1}}^{2})+a(\Vert v\left(  t\right)  \Vert_{H_{0}^{1}}%
^{2})\right)  \Vert u\left(  t\right)  \Vert_{H_{0}^{1}}^{2}\Vert v\left(
t\right)  \Vert_{H_{0}^{1}}^{2}+a(\Vert v\left(  t\right)  \Vert_{H_{0}^{1}%
}^{2})\Vert v\left(  t\right)  \Vert_{H_{0}^{1}}^{2}\nonumber\\
&  =\left(  a(\Vert u\left(  t\right)  \Vert_{H_{0}^{1}}^{2})\Vert u\left(
t\right)  \Vert_{H_{0}^{1}}^{2}-a(\Vert v\left(  t\right)  \Vert_{H_{0}^{1}%
}^{2})\Vert v\left(  t\right)  \Vert_{H_{0}^{1}}^{2}\right)  \left(  \Vert
u\left(  t\right)  \Vert_{H_{0}^{1}}^{2}-\Vert v\left(  t\right)  \Vert
_{H_{0}^{1}}^{2}\right)  \geq0, \label{1.37}%
\end{align}
where we have used (\ref{1.36}) in the last inequality.

Hence, from (\ref{1.37}) and $f^{\prime}\left(  s\right)  \leq\eta$, we infer
\[
\frac{1}{2}\frac{d}{dt}\Vert u-v\Vert_{L^{2}}^{2}\leq\int_{\Omega}\left(
f(u)-f(v)\right)  (u-v)dx=\int_{\Omega}\left(  \int_{v}^{u}f^{\prime
}(s)ds\right)  (u-v)dx\leq\eta\Vert u-v\Vert_{L^{2}}^{2}.
\]
By Remark \ref{Deriv} it is correct to apply Gronwall's lemma over an
arbitrary interval $\left(  \varepsilon,t\right)  $, so
\[
\Vert u(t)-v(t)\Vert_{L^{2}}^{2}\leq\Vert u\left(  \varepsilon\right)
-v\left(  \varepsilon\right)  \Vert_{L^{2}}^{2}\ e^{2\eta(t-\varepsilon
)},\quad t\geq0.
\]
Since Lemma \ref{ContSol} implies that $u,v\in C([0,T],L^{2}\left(
\Omega\right)  ),$ we pass to the limit as $\varepsilon\rightarrow0$ to get%
\[
\Vert u(t)-v(t)\Vert_{L^{2}}^{2}\leq\Vert u\left(  0\right)  -v\left(
0\right)  \Vert_{L^{2}}^{2}\ e^{2\eta t},\quad t\geq0.
\]
Hence, the uniqueness follows.

If $M_{2}=0$ in (\ref{7}), then by (\ref{6b}) the above argument is valid for
weak solutions as well.

The proof of the last statement is the same with the only difference that
condition (\ref{7}) is not needed.
\end{proof}

\section{Existence and structure of attractors}

In this section we will prove the existence of global attractors for the
semiflows generated by regular and strong solutions under different
assumptions in the autonomous case, that is, when the function $h$ does depend
on $t$. We will also establish that the attractors is equal to the unstable
manifold of the set of stationary points.

We consider the following condition instead of (\ref{h}):%
\begin{equation}
h\in L^{2}\left(  \Omega\right)  . \label{h2}%
\end{equation}

Throughout this section, for a metric space $X$ with metric $d$ we will denote
by $dist_{X}\left(  C,D\right)  $ the \ Hausdorff semidistance from $C$ to
$D$, that is, $dist_{X}(C,D)=\sup_{c\in C}\inf_{d\in D}\rho\left(  c,d\right)
.$

\subsection{Regular solutions}

We split this part into three subsections.

\subsubsection{The case of uniqueness\label{SectionAttr1}}

If we assume conditions (\ref{Cont})-(\ref{3}), (\ref{7}), (\ref{1.36}),
(\ref{h2}), $f\in C^{1}(\mathbb{R})$ and$\ f^{\prime}(s)\leq\eta$ , then by
Theorems \ref{existenceweakregularsolutionu0L2} and \ref{uniqueness} we can
define the following continuous semigroup $T_{r}:\mathbb{R}^{+}\times
L^{2}(\Omega)\rightarrow L^{2}(\Omega):$
\begin{equation}
T_{r}(t,u_{0})=u(t), \label{Semigroup}%
\end{equation}
where $u\left(  \text{\textperiodcentered}\right)  $ is the unique regular
solution to (\ref{1}). We denote by $\mathfrak{R}$ the set of fixed points of
$T_{r}$, that is, the points $z$ such that $T_{r}(t,z)=z$ for any $t\geq0$.

We also observe that using the calculations in (\ref{47})-(\ref{50}) in the
Galerkin approximations of any regular solution $u\left(
\text{\textperiodcentered}\right)  $ one can obtain that $u\in L^{\infty
}\left(  \varepsilon,T;L^{p}\left(  \Omega\right)  \right)  ,$ for all
$0<\varepsilon<T$, and then $u\in C_{w}((0,+\infty),L^{p}\left(
\Omega\right)  ).$

Our first purpose is to obtain a global attractor. We recall that the set
$\mathcal{A}$ is a global attractor for $S$ if it is compact, invariant (which
means $T_{r}(t,\mathcal{A})=\mathcal{A}$ for any $t\geq0$) and it attracts any
bounded set $B$, that is,%
\[
dist_{L^{2}}\left(  T_{r}(t,B\right)  ,\mathcal{A})\rightarrow0\text{ as
}t\rightarrow+\infty.
\]

\begin{proposition}
\label{abosorbinginL2firstcase}Let (\ref{Cont})-(\ref{3}), (\ref{7}),
(\ref{1.36}) and (\ref{h2}) hold. Then the semigroup $S$ has a bounded
absorbing set in $L^{2}$; that is, there exists a constant $K$ such that for
any $R>0$ there is a time $t_{0}=t_{0}(R)$ such that
\begin{equation}
\Vert u(t)\Vert_{L^{2}}\leq K\quad\text{ for all }\quad t\geq t_{0},
\label{9.1}%
\end{equation}
where $\left\Vert u_{0}\right\Vert _{L^{2}}\leq R$, $u\left(  t\right)
=T_{r}(t,u_{0}).$ Moreover, there is a constant $L$ such that
\begin{equation}
\int_{t}^{t+1}\Vert u(s)\Vert_{H_{0}^{1}}^{2}ds\leq L\quad\text{ for all
}\quad t\geq t_{0}. \label{9.2}%
\end{equation}

\end{proposition}

\begin{proof}
Multiplying equation (\ref{1}) by $u$ and using (\ref{3}) and Remark
\ref{Deriv} we get
\begin{equation}
\frac{1}{2}\frac{d}{dt}\Vert u(t)\Vert_{L^{2}}^{2}+\frac{m}{2}\Vert
u(t)\Vert_{H_{0}^{1}}^{2}+\alpha_{1}\Vert u(t)\Vert_{L^{P}}^{p}\leq
\kappa|\Omega|+\frac{1}{2\lambda_{1}m}\Vert h\Vert_{L^{2}}^{2}=\frac
{\kappa_{1}}{2}. \label{9.3}%
\end{equation}
Gronwall's lemma and the inequality $\Vert u(t)\Vert_{H_{0}^{1}}^{2}%
\geq\lambda_{1}\Vert u(t)\Vert_{L^{2}}^{2}$ give
\[
\Vert u(t)\Vert_{L^{2}}^{2}\leq\Vert u(\varepsilon)\Vert_{L^{2}}%
^{2}e^{-\lambda_{1}m(t-\varepsilon)}+\frac{\kappa_{1}}{\lambda_{1}m},\text{
for any }\varepsilon>0.
\]
As $u\in C([0,T],L^{2}\left(  \Omega\right)  $ by \ref{ContSol}, pasing to the
limit we have%
\begin{equation}
\Vert u(t)\Vert_{L^{2}}^{2}\leq\Vert u(0)\Vert_{L^{2}}^{2}e^{-\lambda_{1}%
mt}+\frac{\kappa_{1}}{\lambda_{1}m}. \label{9.4}%
\end{equation}
Hence, taking
\[
t\geq t_{0}\equiv\frac{1}{\lambda_{1}m}\ln\left(  {\frac{\lambda_{1}mR}%
{\kappa_{1}}}\right)
\]
we get (\ref{9.1}) for $K=\frac{2\kappa_{1}}{\lambda_{1}m}$. On the other
hand, integrating (\ref{9.3}) between $t$ and $t+1$ and using (\ref{9.4}) we
obtain
\[
m\int_{t}^{t+1}\Vert u(s)\Vert_{H_{0}^{1}}^{2}ds\leq\Vert u(t)\Vert_{L^{2}%
}^{2}+\kappa_{1}\leq
\]
and using the previous bound we get
\[
\int_{t}^{t+1}\Vert u(s)\Vert_{H_{0}^{1}}^{2}ds\leq\frac{\kappa_{1}}{m}%
+\frac{2\kappa_{1}}{\lambda_{1}m},\quad\text{ for all }t\geq t_{0},
\]
so that (\ref{9.2}) follows.
\end{proof}

\bigskip

\begin{proposition}
\label{absorbinginH1case1}Let (\ref{Cont})-(\ref{3}), (\ref{7}), (\ref{1.36})
and (\ref{h2}) hold. Then there exists a bounded absorbing set in $H_{0}%
^{1}\left(  \Omega\right)  \cap L^{p}\left(  \Omega\right)  $; that is, there
is a constant $M$ such that for any $R>0$ there is a time $t_{1}=t_{1}(R)$
such that
\[
\Vert u(t)\Vert_{H_{0}^{1}}+\left\Vert u\left(  t\right)  \right\Vert _{L^{p}%
}\leq M\quad\text{for all }t\geq t_{1},
\]
where $\left\Vert u_{0}\right\Vert _{L^{2}}\leq R$, $u\left(  t\right)
=T_{r}(t,u_{0}).$
\end{proposition}

\begin{proof}
The following calculations are formal but can be justified by Galerking
approximations. Using Proposition \ref{abosorbinginL2firstcase} and arguing as
in(\ref{47})-(\ref{50}) we obtain the existence of a constant $C$ such that
\[
\Vert T_{r}(1,u\left(  0\right)  )\Vert_{H_{0}^{1}}^{2}+\left\Vert
T_{r}\left(  1,u\left(  0\right)  \right)  \right\Vert _{L^{p}}^{p}\leq
C(1+\Vert u(0)\Vert_{L^{2}}^{2}).
\]
Hence, the semigroup property $T_{r}(t+1,u_{0})=T_{r}(1,T_{r}(t,u_{0}))$ and
(\ref{9.1}) imply that
\[
\Vert T_{r}(t+1,u_{0})\Vert_{H_{0}^{1}}^{2}+\left\Vert T_{r}\left(
t+1,u_{0}\right)  \right\Vert _{L^{p}}^{p}\leq C(1+K^{2})\text{ }\forall t\geq
t_{0}\left(  R\right)  ,
\]
if $\left\Vert u_{0}\right\Vert _{L^{2}}\leq R$, which proves the statement.
\end{proof}

\begin{theorem}
\label{ExistAttr2}Let (\ref{Cont})-(\ref{3}), (\ref{7}), (\ref{1.36}) and
(\ref{h2}) hold. Then the equation (\ref{1}) has a connected global attractor
$\mathcal{A}_{r}$, which is bounded in $H_{0}^{1}\left(  \Omega\right)  \cap
L^{p}\left(  \Omega\right)  $.
\end{theorem}

\begin{proof}
Since a bounded set in $H_{0}^{1}(\Omega)$ is relatively compact in
$L^{2}(\Omega)$ which is a connected space, the result follows from Theorem
10.5 in \cite{robinson} and Proposition \ref{absorbinginH1case1}.
\end{proof}

\bigskip

In order to get a characterization of the global attractor in terms of the
unstable manifold of the set of stationary points we need to obtain its
boundedness in the spaces $L^{\infty}\left(  \Omega\right)  $ and
$H^{2}\left(  \Omega\right)  $. This is necessary to define properly a
Lyapunov function.

First, we recall that a function $\phi:\mathbb{R}\rightarrow L^{2}\left(
\Omega\right)  $ is a complete trajectory of the semigroup $S$ if $\phi\left(
t\right)  =T_{r}(t-s,\phi\left(  s\right)  )$ for any $t\geq s$. $\phi$ is
bounded if the set $\cup_{s\in\mathbb{R}}\phi\left(  s\right)  $ is bounded.
It is well known \cite{Lad5} that the global attractor is characterized by%
\begin{equation}
\mathcal{A}_{r}=\{\phi\left(  0\right)  :\phi\text{ is a bounded complete
trajectory}\}. \label{Characterization}%
\end{equation}

\begin{theorem}
\label{boundinLinfty}Let (\ref{Cont})-(\ref{3}), (\ref{7}), (\ref{1.36}) and
(\ref{h2}) hold. Then the global attractor $\mathcal{A}_{r}$ is bounded in
$L^{\infty}(\Omega)$, provided that $h\in L^{\infty}(\Omega)$.
\end{theorem}

\begin{proof}
We define $v_{+}=\max\{v,0\},\ v_{-}=-\max\{-v,0\}.$ We multiply equation
(\ref{1}) by $(u-M)_{+}$ for some appropiate constant $M$ and integrate over
$\Omega$ to obtain
\[
\frac{1}{2}\frac{d}{dt}\int_{\Omega}|(u-M)_{+}|^{2}dx+a(\Vert u(t)\Vert
_{H_{0}^{1}}^{2})\int_{\Omega}|\nabla(u-M)_{+}|^{2}dx=\int_{\Omega
}(f(u(t))+h)(u-M)_{+}dx,
\]
where we have used the equality $\frac{1}{2}\dfrac{d}{dt}\int_{\Omega
}|(u-M)_{+}|^{2}dx=\left(  u_{t},(u-M)_{+}\right)  ,$ which is obtained by regularization.

Since $h\in L^{\infty}(\Omega)$, by (\ref{3}) we deduce that
\[
(f(u)+h)u\leq\widetilde{\kappa}-\widetilde{\alpha}|u|^{p}\text{ for all }%
u\in\mathbb{R}.
\]
It follows that
\[
f(u)+h\leq0\quad\text{ when }\quad u\geq(\frac{\widetilde{\kappa}%
}{\widetilde{\alpha}})^{1/p}=M.
\]
Therefore, taking $u\geq M$ we get
\[
(f(u)+h)(u-M)_{+}=(f(u)+h)u\frac{(u-M)_{+}}{u}=(f(u)+h)u(1-\frac{M}{u}%
)\leq(\widetilde{\kappa}-\widetilde{\alpha}|u|^{p})(1-\frac{M}{u})\leq0.
\]
Thus, by (\ref{2}) and the the Poincar\'{e} inequality, we deduce that
\[
\frac{d}{dt}\int_{\Omega}|(u-M)_{+}|^{2}dx\leq-2m\lambda_{1}\int_{\Omega
}|(u-M)_{+}|^{2}dx.
\]
Using the Gronwall inequality, we have
\[
\int_{\Omega}|(u(t)-M)_{+}|^{2}dx\leq e^{-2m\lambda(t-\tau)}\int_{\Omega
}|(u\left(  \tau\right)  -M)_{+}|^{2}dx.
\]
For any $y\in\mathcal{A}_{r}$ there is by (\ref{Characterization}) a bounded
complete trajectoy $\phi$ such that $\phi\left(  0\right)  =y$. Then taking
$t=0$ and $\tau\rightarrow-\infty$ in the last inequality, we obtain $y\left(
x\right)  =\phi(0,x)\leq M,\ $for a.a. $x\in\Omega$. The same arguments can be
applyied to $(u-M)_{-}$, which shows that
\[
\Vert y\Vert_{L^{\infty}}\leq M,\quad\forall y\in\mathcal{A}_{r}.
\]

\end{proof}

If we assume that $a\in C^{1}(\mathbb{R})$, then it is possible to show that
the global attractor is more regular.

\begin{proposition}
\label{boundinH2}Let (\ref{Cont})-(\ref{3}), (\ref{7}) and (\ref{h2}) hold.
If, additionally, $a\left(  \text{\textperiodcentered}\right)  \in
C^{1}\left(  \mathbb{R}^{+};\mathbb{R}^{+}\right)  $ and $a^{\prime}\left(
s\right)  \geq0$, then there exists an absorbing set in $H^{2}\left(
\Omega\right)  $ and the global attractor is bounded in $H^{2}(\Omega)$.
\end{proposition}

\begin{remark}
$a^{\prime}\left(  s\right)  \geq0$ implies that (\ref{1.36}) holds.
\end{remark}

\begin{proof}
We will prove the existence of an absorbingset in $H^{2}\left(  \Omega\right)
$. The boundedness of the global attractor in this space follows then
immediately. We proceed formally, but the estimates can be justified via
Galerkin approximations.

Let $u(t)=S(t,u_{0})$ with $\left\Vert u_{0}\right\Vert _{L^{2}}\leq R.$
First, we differentiate the equation with respect to $t$
\[
u_{tt}-a^{\prime}(\Vert u\Vert_{H_{0}^{1}}^{2})\frac{d}{dt}\Vert u\Vert
_{H_{0}^{1}}^{2}\Delta u-a(\Vert u\Vert_{H_{0}^{1}}^{2})\Delta u_{t}%
=f^{\prime}(u)u_{t}.
\]
Multiplying by $u_{t}$ we get
\begin{equation}
\frac{1}{2}\frac{d}{dt}\Vert u_{t}\Vert_{L^{2}}^{2}+\frac{1}{2}a^{\prime
}(\Vert u\Vert_{H_{0}^{1}}^{2})(\frac{d}{dt}\Vert u\Vert_{H_{0}^{1}}^{2}%
)^{2}+a(\Vert u\Vert_{H_{0}^{1}}^{2})\Vert u_{t}\Vert_{H_{0}^{1}}^{2}%
=\int_{\Omega}f^{\prime}(u)(u_{t})^{2}dx. \label{equationmultbyut}%
\end{equation}
By (\ref{2}) we obtain
\begin{equation}
\frac{1}{2}\frac{d}{dt}\Vert u_{t}\Vert_{L^{2}}^{2}+m\Vert u_{t}\Vert
_{H_{0}^{1}}^{2}\leq\eta\Vert u_{t}\Vert_{L^{2}}^{2}. \label{h2primera}%
\end{equation}
Second, multiplying (\ref{1}) by $u_{t}$ and reordening terms, we obtain
\begin{equation}
\frac{d}{dt}\left(  \frac{a(\Vert u\Vert_{H_{0}^{1}}^{2})}{2}\Vert
u\Vert_{H_{0}^{1}}^{2}-\int_{\Omega}\mathcal{F}(u)dx-\int_{\Omega
}h(x)udx\right)  +\Vert u_{t}\Vert_{L^{2}}^{2}=\frac{a^{\prime}(\Vert
u\Vert_{H_{0}^{1}}^{2})}{2}\Vert u\Vert_{H_{0}^{1}}^{2}\frac{d}{dt}\Vert
u\Vert_{H_{0}^{1}}^{2}. \label{9.5.2}%
\end{equation}
Proposition \ref{absorbinginH1case1} implies that%
\[
a^{\prime}(\Vert z\Vert_{H_{0}^{1}}^{2})\leq\gamma:=sup_{|s|\leq M}a^{\prime
}(s^{2})
\]
if $z$ belongs to the absorbing set in $H_{0}^{1}\left(  \Omega\right)  $. On
the other hand, multiplying the equation by $-\Delta u$ and using Proposition
\ref{absorbinginH1case1}, we obtain
\[
\frac{d}{dt}\Vert u\Vert_{H_{0}^{1}}^{2}+m\Vert\Delta u(t)\Vert_{L^{2}}%
^{2}\leq\eta\Vert u(t)\Vert_{H_{0}^{1}}^{2}+\frac{1}{m}\Vert h\Vert_{L^{2}%
}^{2}\leq K_{1}\quad\forall t\geq t_{1}(R).
\]
Hence, by (\ref{9.5.2}) and Proposition \ref{absorbinginH1case1}, it follows
\begin{equation}
\frac{d}{dt}\left(  \frac{a(\Vert u\Vert_{H_{0}^{1}}^{2})}{2}\Vert
u\Vert_{H_{0}^{1}}^{2}-\int_{\Omega}\mathcal{F}(u)dx-\int_{\Omega
}h(x)udx\right)  +\Vert u_{t}\Vert_{L^{2}}^{2}\leq\frac{\gamma}{2}K_{1}%
M^{2},\quad\forall t\geq t_{1}(R). \label{h2segunda}%
\end{equation}
Multiplying both sides of the inequality $f^{\prime}(s)\leq\eta$ by $s$ and
integrating between $0$ and $s$, we obtain
\begin{equation}
sf(s)\leq\mathcal{F}(s)+\frac{s^{2}}{2}\eta,\quad\forall s\in\mathbb{R}.
\label{h2tercera}%
\end{equation}
Moreover, integrating $f^{\prime}(s)\leq\eta$ twice between $0$ and $s$, we
get
\begin{equation}
\mathcal{F}(s)\leq\frac{\eta}{2}s^{2}+Cs,\quad\forall s\in\mathbb{R}.
\label{h2cuarta}%
\end{equation}
Now, we multiply (\ref{1}) by $u$, integrate between $t$ and $t+1$ to obtain
\[
\frac{1}{2}\Vert u(t+1)\Vert_{L^{2}}^{2}+\int_{t}^{t+1}\left(  a(\Vert
u\Vert_{H_{0}^{1}}^{2})\Vert u(s)\Vert_{H_{0}^{1}}^{2}-\int_{\Omega
}f(u)udx-\int_{\Omega}h(x)udx\right)  ds=\frac{1}{2}\Vert u(t)\Vert_{L^{2}%
}^{2}.
\]
From (\ref{h2tercera}) and Proposition \ref{abosorbinginL2firstcase} it
follows
\[
\int_{t}^{t+1}\left(  \frac{a(\Vert u\Vert_{H_{0}^{1}}^{2})}{2}\Vert
u\Vert_{H_{0}^{1}}^{2}-\int_{\Omega}\mathcal{F}(u)dx-\int_{\Omega
}h(x)udx\right)  ds\leq\frac{1}{2}\Vert u(t)\Vert_{L^{2}}^{2}+\frac{\eta}%
{2}\int_{t}^{t+1}\Vert u\Vert_{L^{2}}^{2}ds\leq\widetilde{L}\quad\forall t\geq
t_{0}%
\]
The last inequality allows us to apply Uniform Gronwall Lemma to
(\ref{h2segunda}) to obtain
\begin{equation}
\frac{a(\Vert u\Vert_{H_{0}^{1}}^{2})}{2}\Vert u\Vert_{H_{0}^{1}}^{2}%
-\int_{\Omega}\mathcal{F}(u)dx-\int_{\Omega}h(x)udx\leq\widetilde{L}%
+\frac{\gamma}{2}{K}_{1}M^{2}\quad\forall t\geq t_{1}+1. \label{9.10.1}%
\end{equation}
Using (\ref{2}) and (\ref{h2cuarta}) we get
\begin{equation}
\frac{a(\Vert u\Vert_{H_{0}^{1}}^{2})}{2}\Vert u\Vert_{H_{0}^{1}}^{2}%
-\int_{\Omega}\mathcal{F}(u)dx-\int_{\Omega}h(x)udx\geq-\frac{1}{2m\lambda
_{1}}\Vert h\Vert_{L^{2}}^{2}-\frac{\eta}{2}\Vert u\Vert_{L^{2}}%
^{2}-\widetilde{C}\Vert u\Vert_{L^{2}}. \label{9.11}%
\end{equation}
Now, integrating (\ref{h2segunda}) from $t$ to $t+1$, using (\ref{9.10.1}),
(\ref{9.11}), by Proposition \ref{abosorbinginL2firstcase} we have
\begin{equation}
\int_{t}^{t+1}\Vert u_{s}\Vert_{L^{2}}^{2}ds\leq\widetilde{L}+\gamma
K_{1}M^{2}+\frac{1}{2\lambda_{1}m}\Vert h\Vert_{L^{2}}^{2}+\frac{\eta}{2}%
K^{2}+CK=\rho_{1},\quad\forall t\geq t_{1}+1. \label{9.13}%
\end{equation}
Hence, the last equation allow us to apply to (\ref{h2primera}) the Uniform
Gronwall Lemma \cite{Temam88} to obtain
\begin{equation}
\Vert\frac{du}{dt}(t)\Vert_{L^{2}}^{2}\leq\rho_{2},\quad\forall t\geq t_{1}+2.
\label{9.14}%
\end{equation}
Finally, we multiply (\ref{1}) by $-\Delta u$ and use (\ref{2}) to obtain
\[
\frac{m}{2}\Vert\Delta u\Vert_{L^{2}}^{2}\leq\eta\Vert u\Vert_{H_{0}^{1}}%
^{2}+\frac{1}{m}\Vert h\Vert_{L^{2}}^{2}+\frac{1}{m}\Vert u_{t}\Vert_{L^{2}%
}^{2}.
\]
Thus, by Proposition \ref{absorbinginH1case1} and (\ref{9.14}), we deduce
that
\[
\Vert u\Vert_{H^{2}}^{2}\leq\rho_{3}\quad\forall t\geq t_{1}+2.
\]

\end{proof}

\subsubsection{Abstract theory of attractors for multivalued
semiflows\label{Abstract}}

Prior to studying the case of non-uniqueness, we recall some well-known
results concerning the structure of attractors for multivalued semiflows.

Consider a metric space $(X,d)$ and a family of functions $\mathcal{R}%
\subset\mathcal{C}(\mathbb{R}_{+};X)$. Denote by $P(X)$ the class of nonempty
subsets of $X$. Then we define the multivalued map $G:\mathbb{R}_{+}\times
X\rightarrow P(X)$ associated with the family $\mathcal{R}$ as follows
\begin{equation}
G(t,u_{0})=\{u(t):u(\cdot)\in\mathcal{R},u(0)=u_{0}\}. \label{defG}%
\end{equation}

In this abstract setting, the multivalued map $G$ is expected to satisfy some
properties that fit in the framework of multivalued dynamical systems. The
first concept is given now.

\begin{definition}
A multivalued map $G:\mathbb{R}_{+}\times X\rightarrow P(X)$ is a multivalued
semiflow (or m-semiflow) if $G(0,x)=x$ for all $x\in X$ and $G(t+s,x)\subset
G(t,G(s,x))$ for all $t,s\geq0$ and $x\in X$. \newline If the above is not
only an inclusion, but an equality, it is said that the m-semiflow is strict.
\end{definition}

Once a multivalued semiflow is defined, we recall the following concepts.

\begin{definition}
A map $\gamma:\mathbb{R}\rightarrow X$ is called a complete trajectory of
$\mathcal{R}$ (resp. of $G$) if $\gamma(\cdot+h)\mid_{\lbrack0,\infty)}%
\in\mathcal{R}$ for all $h\in\mathbb{R}$ (resp. if $\gamma(t+s)\in
G(t,\gamma(s))$ for all $s\in\mathbb{R}$ and $t\geq0)$.

A point $z \in X$ is a fixed point of $\mathcal{R}$ if $\varphi(\cdot)\equiv z
\in\mathcal{R}$. The set of all fixed points will be denoted by $\mathfrak{R}%
_{\mathcal{R}}$.

A point $z\in X$ is a stationary point of $G$ if $z \in G(t,z)$ for all
$t\geq0$).
\end{definition}

\begin{definition}
Given an m-semiflow $G$ a set $B\subset X$ is said to be negatively
(positively) invariant if $B\subset G(t,B)$ ($G(t,B)\subset B$) for all
$t\geq0$, and strictly invariant (or, simply, invariant) if it is both
negatively and positively invariant.

The set $B$ is said to be weakly invariant if for any $x\in B$ there exists a
complete trajectory $\gamma$ of $\mathcal{R}$ contained in $B$ such that
$\gamma(0)=x$. We observe that weak invariance implies negative invariance.
\end{definition}

\begin{definition}
A set $\mathcal{A}\subset X$ is called a global attractor for the m-semiflow
$G$ if it is negatively invariant and it attracts all bounded subsets, i.e.,
$dist_{X}(G(t,B),\mathcal{A})\rightarrow0$ as $t\rightarrow+\infty$.
\end{definition}

\begin{remark}
When $\mathcal{A}$ is compact, it is the minimal closed attracting set
\cite[Remark 5]{melnikvalero}.
\end{remark}

In order to obtain a detailed characterization of the internal structure of a
global attractor, we introduce an axiomatic set of properties on the set
$\mathcal{R}$.

\begin{itemize}
\item[(K1)] For any $x \in X$ there exists at least one element $\varphi
\in\mathcal{R}$ such that $\varphi(0) =x$.

\item[(K2)] $\varphi_{\tau} (\cdot) := \varphi(\cdot+\tau) \in\mathcal{R}$ for
any $\tau\geq0$ and $\varphi\in\mathcal{R}$ (translation property).

\item[(K3)] Let $\varphi_{1}, \varphi_{2} \in\mathcal{R}$ be such that
$\varphi_{2}(0)=\varphi_{1}(s)$ for some $s>0$. Then, the function $\varphi$
defined by
\[
\varphi(t) = \left\lbrace
\begin{array}
[c]{lcc}%
\varphi_{1}(t) \quad0 \leq t \leq s, &  & \\
\varphi_{2}(t-s) \quad s \leq t, &  &
\end{array}
\right.
\]
belongs to $\mathcal{R}$ (concatenation property).

\item[(K4)] For any sequence $\{\varphi^{n}\} \subset\mathcal{R}$ such that
$\varphi^{n}(0) \rightarrow x_{0}$ in X, there exist a subsequence
$\{\varphi^{n_{k}}\}$ and $\varphi\in\mathcal{R}$ such that $\varphi^{n_{k}%
}(t) \rightarrow\varphi(t)$ for all $t \geq0$.
\end{itemize}

\begin{remark}
If in assumption $(K1)$, for every $x\in X$, there exists a unique $\varphi
\in\mathcal{R}$ such that $\varphi(0)=x$, then the set $\{\varphi
\in\mathcal{R} : \varphi(0)=x\}$ consists of a single trajectory $\varphi$,
and the equality $G(t, x) = \varphi(t)$ defines a classical semigroup
$G:\mathbb{R}^{+}\times X \rightarrow X.$
\end{remark}

It is immediate to observe \cite[Proposition 2]{caraballorubio} or \cite[Lemma
9]{kapustyanpankov} that $\mathcal{R}$ fulfilling (K1) and (K2) gives rise to
an m-semiflow $G$ through (\ref{defG}), and if besides (K3) holds, then this
m-semiflow is strict. In such a case, a global bounded attractor, supposing
that it exists, is strictly invariant \cite[Remark 8]{melnikvalero}.

Several properties concerning fixed points, complete trajectories and global
attractors are summarized in the following results \cite{kapustyankasyanov}.

\begin{lemma}
Let (K1)-(K2) be satisfied. Then every fixed point (resp. complete trajectory)
of $\mathcal{R}$ is also a fixed point (resp. complete trajectory) of $G$.

If $\mathcal{R}$ fulfills (K1)-(K4), then the fixed points of $\mathcal{R}$
and $G$ coincide. Besides, a map $\gamma:\mathbb{R}\rightarrow X$ is a
complete trajectory of $\mathcal{R}$ if and only if it is continuous and a
complete trajectory of $G$.
\end{lemma}

The standard well-known result in the single-valued case for describing the
attractor as the union of bounded complete trajectories (see
(\ref{Characterization})) reads in the multivalued case as follows.

\begin{theorem}
\label{structureattractor} Consider $\mathcal{R}$ satisfying (K1) and (K2),
and either (K3) or (K4). Assume also the conditions in Theorem
\ref{compactfixedpoints} and that $G$ possesses a compact global attractor
$\mathcal{A}$. Then
\begin{equation}
\mathcal{A}=\{\gamma(0):\gamma\in\mathbb{K}\}=\cup_{t\in\mathbb{R}}%
\{\gamma(t):\gamma\in\mathbb{K}\}, \label{attractor}%
\end{equation}
where $\mathbb{K}$ denotes the set of all bounded complete trajectories in
$\mathcal{R}$.
\end{theorem}

We finish this section by stating a general result about the existence of
attracttors. We recall that the map $t\mapsto G(t,x)$ is upper semicontinuous
if for any $x\in X$ and any neighborhood $O(G_{r}(t,x))$ in $X$ there exists
$\delta>0$ such that if $d(y,x)<\delta$, then $G(t,y)\subset O$.

\begin{theorem}
\label{AttrExist}\cite[Theorem 4 and Remark 8]{melnikvalero}Let the map
$t\mapsto G(t,x)$ be upper semicontinuous with closed values. If there exists
a compact attracting set $K$, that is,
\[
dist_{X}(G(t,B),K)\rightarrow0,\text{ as }t\rightarrow+\infty,
\]
for any bounded set $B$, then $G$ possesses a global compact attractor
$\mathcal{A}$, which is the minimal closed attracting set. If, moreover, $G$
is strict, then $\mathcal{A}$ is invariant.
\end{theorem}

We observe that, although in the papers \cite{melnikvalero},
\cite{kapustyankasyanov} the space $X$ is assumed to be complete, the results
are true in a non-complete space.

\subsubsection{The case of non-uniqueness\label{SectionNonuniqueness}}

If we do not assume the additional assumptions on the function $a$ of the
previous section ensuring uniqueness of the Cauchy problem, we we have to
define a multivalued semiflow.

We have two possibilites:\ either to consider the conditions of Theorem
\ref{existenceweakregularsolutionu0L2} or to use the conditions of Theorem
\ref{existenceweaksolutionu0L2sinderivadaversion2014}.

If we assume conditions (\ref{Cont})-(\ref{3}), (\ref{7}), (\ref{h2}) and
(\ref{Condp}), then by Theorem
\ref{existenceweaksolutionu0L2sinderivadaversion2014} for any $u_{0}\in
L^{2}\left(  \Omega\right)  $ there exists at least one regular solution and
(\ref{Acotacionfu}) implies that $f(u)\in L^{2}(\varepsilon,T;L^{2}(\Omega))$
for any regular solution, so $\dfrac{du}{dt}\in L^{2}(\varepsilon
,T;L^{2}(\Omega))$ as well. In this case, as $H_{0}^{1}\left(  \Omega\right)
\subset L^{p}\left(  \Omega\right)  $, we have that $u\in C((0,+\infty
),H_{0}^{1}\left(  \Omega\right)  )\subset C\left(  \left(  0,+\infty\right)
,L^{p}\left(  \Omega\right)  \right)  .$

If we assume that $f\in C^{1}(\mathbb{R})$ is such that $f^{\prime}(s)\leq
\eta$ and conditions (\ref{Cont})-(\ref{3}), (\ref{7}) and (\ref{h2}) as well,
then we known by Theorem \ref{existenceweakregularsolutionu0L2} that for any
$u_{0}\in L^{2}\left(  \Omega\right)  $ there exists at least one regular
solution. In order to obtain the necessary estimates leading to the existence
of a global attractor, we need to ensure that
\begin{equation}
\dfrac{du}{dt}\in L^{2}(\varepsilon,T;L^{2}(\Omega)),\text{ for all
}0<\varepsilon<T, \label{CondDerivu}%
\end{equation}
hods. For this aim we will restrict the set of regular solutions to the ones
satisfying this property. We note that the set of regular solutions of that
kind is non-empty, as using inequalities (\ref{47})-(\ref{51}) in the proof of
Theorem \ref{existenceweakregularsolutionu0L2} we prove that the regular
solution satisfies (\ref{CondDerivu}). By (\ref{CondDerivu}) we use
(\ref{sellandyou1}) and argue as in (\ref{47})-(\ref{50}) to show that $u\in
L^{\infty}(\varepsilon,T;L^{p}(\Omega))$, so $u\in C_{w}\left(  \left(
0,+\infty\right)  ,L^{p}\left(  \Omega\right)  \right)  $ for any regular
solution. Also, by \cite[p.102]{sellyou} we deduce that $u\in C((0,+\infty
),H_{0}^{1}\left(  \Omega\right)  )$.

We also observe that we can to force all the regular solutions to satisfy
$\dfrac{du}{dt}\in L^{2}(\varepsilon,T;L^{2}(\Omega))$ with an additional
assumption on the constant $p$, which is weaker than (\ref{Condp}). This is
achieved by obtaining that $f(u)\in L^{2}(\varepsilon,T;L^{2}(\Omega))$, which
can be done by using an interpolation inequality. Indeed, for $u\in L^{\infty
}(\varepsilon,T;H_{0}^{1}(\Omega))\cap L^{2}(\varepsilon,T;D(A))$ we have the
interpolation inequality%
\begin{equation}
\Vert u\Vert_{L^{2(\gamma+1)}(\varepsilon,T;L^{2(\gamma+1)}(\Omega
))}^{2(\gamma+1)}\leq\Vert u\Vert_{L^{\infty}(\varepsilon,T;L^{p_{1}}%
(\Omega))}^{2\gamma}\Vert u\Vert_{L^{2}(\varepsilon,T;L^{p_{2}}(\Omega))}^{2},
\label{interpol}%
\end{equation}
where $\gamma=\frac{4}{n-2},\ p_{1}=\frac{2n}{n-2},\ p_{2}=\frac{2n}{n-4}$,
provided that $n>4$; $\gamma<2,\ p_{1}=4,\ p_{2}=\frac{4}{2-\gamma}$ if
$n=4;\ \gamma=3,\ p_{1}=6,\ p_{2}=+\infty$ if $n=3;\ $and $\gamma\geq0$ is
arbitrary for $n=1,2.$ We have used the embeddings $H_{0}^{1}\left(
\Omega\right)  \subset L^{p_{1}}\left(  \Omega\right)  ,\ H^{2}\left(
\Omega\right)  \subset L^{p_{2}}\left(  \Omega\right)  $ and \cite[Lemma
II.4.1, p. 72]{Werner}. Thus, (\ref{1.4}) implies that $f(u)\in L^{2}%
(\varepsilon,T;L^{2}(\Omega))$ if
\begin{equation}
p\leq\gamma+2 \label{Condp2}%
\end{equation}
and also that%
\begin{equation}
\Vert f(u)\Vert_{L^{2}(\varepsilon,T;L^{2}(\Omega))}^{2}=\int_{\varepsilon
}^{T}\int_{\Omega}|f(u(x,t))|^{2}dxdt\leq C_{1}+C_{2}\int_{\varepsilon}%
^{T}\int_{\Omega}|u(x,t)|^{2(\gamma+1)}dxdt. \label{Acotfu}%
\end{equation}

Let
\[
\mathcal{R}=K_{r}^{+}:=\{u(\cdot):u\text{ is a regular solution of (\ref{1})
satisfying (\ref{CondDerivu})}\}.
\]

\begin{remark}
If condition (\ref{Condp2}) is satisfied, then $K_{r}^{+}$ contains all the
regular solutions.
\end{remark}

We define the (possibly multivalued) map $G_{r}:\mathbb{R}^{+}\times
L^{2}(\Omega)\rightarrow P(L^{2}(\Omega))$ by
\[
G_{r}(t,u_{0})=\{u(t):u\in K_{r}^{+}\text{ and }u(0)=u_{0}\}.
\]
With respect to the axiomatic properties $\left(  K1\right)  -\left(
K4\right)  $ given above, we observe that obviously $\left(  K1\right)  $ is
true, and $\left(  K2\right)  $ can be proved easily using equality
(\ref{EquationRegular}). Therefore, $G_{r}$ is a multivalued semiflow by the
results of the previous section. In this case we are not able to prove
$\left(  K3\right)  $, so $G_{r}$ could be non-strict. Further we will prove
that $\left(  K4\right)  $ holds true.

\begin{lemma}
\label{lemma1} Let assume (\ref{Cont})-(\ref{3}), (\ref{7}) and (\ref{h2}).
Additionally, assume one of the following assumptions:

\begin{enumerate}
\item $f\in C^{1}(\mathbb{R})$ is such that $f^{\prime}(s)\leq\eta$;

\item (\ref{Condp}) is true.
\end{enumerate}

Given a sequence $\{u^{n}\}\subset K_{r}^{+}$ such that $u^{n}(0)\rightarrow
u_{0}$ weakly in $L^{2}(\Omega)$, there exists a subsequence of $\{u^{n}\}$
(relabeled the same) and $u\in K_{r}^{+}$, satisfying $u(0)=u_{0}$, such that
\[
u^{n}(t)\rightarrow u(t)\mathit{\ }\text{strongly in }H_{0}^{1}(\Omega
)\quad\forall t>0.
\]

\end{lemma}

\begin{proof}
We take an arbitrary $T>0$. Arguing as in the proof of Theorem
\ref{existenceweakregularsolutionu0L2} we obtain the existence of a
subsequence of $u^{n}$ such that
\begin{equation}%
\begin{split}
\{u^{n}\}  &  \text{ is bounded in }L^{\infty}(0,T;L^{2}(\Omega)),\\
\{u^{n}\}  &  \text{ is bounded in }L^{p}(0,T;L^{p}(\Omega)),\\
\{f(u^{n})\}  &  \text{ is bounded in }L^{q}(0,T;L^{q}(\Omega)),
\end{split}
\label{aubinlionsaux}%
\end{equation}
The only difference is that we obtain inequality (\ref{1.10.1}) in an arbitray
interval $\left(  \varepsilon,T\right)  $ and then pass to the limit as
$\varepsilon\rightarrow0$ (see the proof of Proposition
\ref{abosorbinginL2firstcase}).

Since $\dfrac{du^{n}}{dt}\in L^{2}(\varepsilon,T;L^{2}(\Omega)),\ $for any
$\varepsilon>0$, we have by \cite[p.102]{sellyou} that
\begin{equation}
\frac{d}{dt}\Vert u^{n}\Vert_{H_{0}^{1}}^{2}=2(-\Delta u^{n},u_{t}^{n})\text{
for a.a. }t. \label{sellandyou1}%
\end{equation}
and $u\in C((0,T],H_{0}^{1}\left(  \Omega\right)  )$. Also, as $f\left(
u^{n}\right)  \in L^{2}(\varepsilon,T;L^{2}(\Omega)),$ by regularization one
can show that $\left(  F\left(  u^{n}\left(  t\right)  \right)  ,1\right)  $
an absolutely continuous function on $[\varepsilon,T]$ for any $\varepsilon>0$
and%
\[
\frac{d}{dt}\left(  F\left(  u^{n}\left(  t\right)  \right)  ,1\right)
=\left(  f\left(  u^{n}\left(  t\right)  \right)  ,\frac{du^{n}}{dt}\right)
\text{ for a.a. }t>0.
\]

Therefore, arguing as in the proofs of Theorems
\ref{existenceweakregularsolutionu0L2} and
\ref{existenceweaksolutionu0L2sinderivadaversion2014} there exists $u\in
L^{\infty}(\varepsilon,T;L^{2}(\Omega))\cap L^{2}(0,T;H_{0}^{1}(\Omega))$ and
a subsequence $u^{n}$, relabelled the same, such that%
\begin{equation}%
\begin{split}
u_{n}  &  \overset{\ast}{\rightharpoonup}u\text{ in }L^{\infty}(0,T;L^{2}%
(\Omega))\\
u_{n}  &  \overset{\ast}{\rightharpoonup}u\text{ in }L^{\infty}(\varepsilon
,T;H_{0}^{1}(\Omega))\\
u_{n}  &  \rightharpoonup u\text{ in }L^{2}(0,T;H_{0}^{1}(\Omega))\\
u_{n}  &  \rightharpoonup u\text{ in }L^{p}(0,T;L^{p}(\Omega))\\
u_{n}  &  \rightharpoonup u\text{ in }L^{2}(\varepsilon,T;D(A)),\\
\frac{du_{n}}{dt}  &  \rightharpoonup\frac{du}{dt}\text{ in }L^{2}%
(\varepsilon,T;L^{2}(\Omega))\\
f(u_{n})  &  \rightharpoonup f(u)\text{ in }L^{q}(0,T;L^{q}(\Omega)),\\
f(u_{n})  &  \rightharpoonup f(u)\text{ in }L^{2}(\varepsilon,T;L^{2}\left(
\Omega)\right)  ,\\
a(\Vert u_{n}\Vert_{H_{0}^{1}}^{2})\Delta u_{n}  &  \rightharpoonup a(\Vert
u\Vert_{H_{0}^{1}}^{2})\Delta u\text{ in }L^{2}(\varepsilon,T;L^{2}(\Omega)).
\end{split}
\label{10.2}%
\end{equation}
In view of (\ref{10.2}), the Aubin-Lions Compactness Lemma gives
\begin{equation}
u_{n}\rightarrow u\text{ in }L^{2}(\varepsilon,T;H_{0}^{1}(\Omega)).
\label{useful}%
\end{equation}
Since the sequence $\{u^{n}\}$ is equicontinuous in $L^{2}(\Omega)$ on
$[\varepsilon,T]$ and bounded in $C([\varepsilon,T],H_{0}^{1}(\Omega))$, by
the compact embedding $H_{0}^{1}(\Omega)$ $\subset L^{2}(\Omega)$ and the
Ascoli-Arzel\`{a} theorem, a subsequence fulfills
\[
u^{n}\rightarrow u\text{ in }C([\varepsilon,T],L^{2}(\Omega)),
\]%
\[
u^{n}(t)\rightharpoonup u(t)\text{ in }H_{0}^{1}(\Omega)\quad\forall
t\in\lbrack\varepsilon,T].
\]

By a similar argument as in the proof of Theorem
\ref{existenceweakregularsolutionu0L2} we establish that $u\left(  0\right)
=u_{0}.$

Finally, we shall prove that $u^{n}(t)\rightarrow u(t)$ in $H_{0}^{1}(\Omega)$
for all $t\in\lbrack\varepsilon,T]$.

Multiplying (\ref{1}) by $u_{t}^{n}$ and using (\ref{funcionA}),
(\ref{sellandyou1}) and (\ref{10.2}) we obtain
\[
\left\Vert \frac{du^{n}}{dt}\right\Vert _{L^{2}}^{2}+\frac{d}{dt}\left(
\frac{1}{2}A(\Vert u^{n}(t)\Vert_{H_{0}^{1}}^{2}\right)  \leq C_{\varepsilon
}.
\]
Thus,
\[
A(\Vert u^{n}(t)\Vert_{H_{0}^{1}}^{2})\leq A(\Vert u^{n}(s)\Vert_{H_{0}^{1}%
}^{2})+C_{\varepsilon}(t-s),\ t\geq s\geq\varepsilon.
\]
The same inequality is valid for the limit function $u(\cdot).$ Hence, the
functions $J_{n}(t)=A(\Vert u^{n}(t)\Vert_{H_{0}^{1}}^{2})-C_{\varepsilon}t,$
$J(t)=A(\Vert u(t)\Vert_{H_{0}^{1}}^{2})-C_{\varepsilon}t$ are continuous and
non-increasing in $[\varepsilon,T].$ Moreover, from (\ref{useful}) we deduce
that $J_{n}(t)\rightarrow J(t)$ for a.e. $t\in(\varepsilon,T).$ Take
$\varepsilon<t_{m}<T$ such that $t_{m}\rightarrow T$ and $J_{n}(t_{m}%
)\rightarrow J(t_{m})$ for all $m.$ Then
\[
J_{n}(T)-J(T)\leq J_{n}(t_{m})-J(T)\leq|J_{n}(t_{m})-J(t_{m})|+|J(t_{m}%
)-J(T)|.
\]
For any $\delta>0$ there exist $m(\delta)$ and $N(\delta)$ such that
$J^{n}(T)-J(T)\leq\delta$ if $n\geq N.$ Then $\limsup J_{n}(T)\leq J(T),$ so
$\limsup\Vert u^{n}(T)\Vert_{H_{0}^{1}}^{2}\leq\Vert u(T)\Vert_{H_{0}^{1}}%
^{2}.$ As $u^{n}(T)\rightarrow u(T)$ weakly in $H_{0}^{1}$ implies
$\liminf\Vert u^{n}(T)\Vert_{H_{0}^{1}}^{2}\geq\Vert u(T)\Vert_{H_{0}^{1}}%
^{2},$ we obtain
\[
\Vert u^{n}(T)\Vert_{H_{0}^{1}}^{2}\rightarrow\Vert u(T)\Vert_{H_{0}^{1}}%
^{2},
\]
so that $u^{n}(T)\rightarrow u(T)$ strongly in $H_{0}^{1}(\Omega).$ In order
to finish the proof rigorously, we have to justify that $\limsup J_{n}(T)\leq
J(T)$ implies $\limsup\Vert u^{n}(T)\Vert_{H_{0}^{1}}^{2}\leq\Vert
u(T)\Vert_{H_{0}^{1}}^{2}.$ But this follows by contradiction using the
continuity of the function $A\left(  s\right)  $.
\end{proof}

\begin{corollary}
\label{PropK4} Assume the conditions of Lemma \ref{lemma1}. Then the set
$K_{r}^{+}$ satisfies condition $(K4)$.
\end{corollary}

\begin{proposition}
\label{uppersemicontinuity} Assume the conditions of Lemma \ref{lemma1}. The
multivalued semiflow $G_{r}$ is upper semicontinuous for all $t\geq0$, that
is, for any neighborhood $O(G_{r}(t,u_{0}))$ in $L^{2}(\Omega)$ there exists
$\delta>0$ such that if $\Vert u_{0}-v_{0}\Vert<\delta$, then $G_{r}%
(t,v_{0})\subset O$. Also, it has compact values.
\end{proposition}

\begin{proof}
We argue by contradiction. Assume that there exists $t\geq0,u_{0}\in
L^{2}(\Omega)$, a neighbourhood $O(G_{r}(t,u_{0}))$ and a sequence $\{y_{n}\}$
which fulfills that each $y_{n}\in G_{r}(t,u_{0}^{n})$, where $u_{0}^{n}$
converges strongly to $u_{0}$ in $L^{2}(\Omega)$ and $y_{n}\notin
O(G_{r}(t,u_{n}))$ for all $n\in\mathbb{N}$. Since $y_{n}\in G_{r}(t,u_{0}%
^{n})$ for all $n$, there exists $u^{n}\in K_{r}(u_{0}^{n})$ such that
$y_{n}=u^{n}(t)$. Now, since $\{u_{0}^{n}\}$ is a convergent sequence of
initial data, making use of Lemma \ref{lemma1} there exists a subsequence of
$\{u^{n}(t)\}$ which converges to a function $u(t)\in G_{r}(t,u_{0})$. This is
a contradiction because $y_{n}\notin O(G_{r}(t,u_{0}))$ for any $n\in
\mathbb{N}$.
\end{proof}

\begin{proposition}
\label{absorbingset} Assume the conditions of Lemma \ref{lemma1}. There exists
an absorbing set $B_{1}$ for $G_{r}$, which is compact in $H_{0}^{1}\left(
\Omega\right)  $ and bounded in $L^{p}\left(  \Omega\right)  $.
\end{proposition}

\begin{proof}
Reasonging as in Proposition \ref{abosorbinginL2firstcase}, we obtain an
absorbing set $B_{0}$ in $L^{2}\left(  \Omega\right)  .$

Let $K>0$ be such that $\left\Vert y\right\Vert \leq K$ for all $y\in B_{0}$.
Since $\dfrac{du}{dt}\in L\left(  \varepsilon,T;L^{2}\left(  \Omega\right)
\right)  $ we are allowed to multiply (\ref{1}) by $u_{t}$, use
(\ref{sellandyou1}) and argue as in (\ref{47})-(\ref{50}) to obtain the
existence of a constant $C$ such that
\begin{equation}
\Vert u\left(  1\right)  \Vert_{H_{0}^{1}}^{2}+\left\Vert u\left(  1\right)
\right\Vert _{L^{p}}^{p}\leq C(1+\Vert u(0)\Vert_{L^{2}}^{2}),
\label{IneqH1Lp}%
\end{equation}
for any regular solution $u\left(  \text{\textperiodcentered}\right)  $ with
initial condition $u\left(  0\right)  $.

For any $u_{0}\in L^{2}\left(  \Omega\right)  $ with $\left\Vert
u_{0}\right\Vert _{L^{2}}\leq R$ and any $u\in K_{r}^{+}$ such that $u\left(
0\right)  =u_{0}$, the semiflow property $G_{r}(t+1,u_{0})\subset
G_{r}(1,G_{r}(t,u_{0}))$ and $G_{r}(t,u_{0})\subset B_{0}$, if $t\geq
t_{0}\left(  R\right)  ,$ imply that
\[
\Vert u\left(  t+1\right)  \Vert_{H_{0}^{1}}^{2}+\left\Vert u\left(
t+1\right)  \right\Vert _{L^{p}}^{p}\leq C(1+K^{2})\text{ }\forall t\geq
t_{0}\left(  R\right)  .
\]
Then there exists $M>0$ such that the closed ball $B_{M}$ in $H_{0}^{1}\left(
\Omega\right)  \cap L^{p}\left(  \Omega\right)  $ centered at $0$ with radius
$M$ is absorbing for $G_{r}$.

By Lemma \ref{lemma1} the set $B_{1}=\overline{G_{r}(1,B_{M})}$ is an
absorbing set which is compact in $H_{0}^{1}\left(  \Omega\right)  $. Also,
inequality (\ref{IneqH1Lp}) implies that $B_{1}$ is bounded in $L^{p}\left(
\Omega\right)  .$
\end{proof}

\begin{theorem}
\label{existenceatracttorcase2}Assume the conditions of Lemma \ref{lemma1}.
Then the multivalued semiflow $G_{r}$ possesses a global compact attractor
$\mathcal{A}_{r}$. Moreover, for any set $B$ bounded in $L^{2}(\Omega)$ we
have
\begin{equation}
dist_{H_{0}^{1}}(G_{r}(t,B),\mathcal{A}_{r})\rightarrow0\quad\text{ as
}t\rightarrow\infty. \label{distancias}%
\end{equation}
Also $\mathcal{A}_{r}$ is compact in $H_{0}^{1}(\Omega)$ and bounded in
$L^{p}\left(  \Omega\right)  $.
\end{theorem}

\begin{proof}
From Propositions \ref{uppersemicontinuity} and \ref{absorbingset} we deduce
that the multivalued semiflow $G_{r}$ is upper semicontinuous with closed
values and the existence of an absorbing which is compact in $H_{0}^{1}\left(
\Omega\right)  $ and bounded in $L^{p}\left(  \Omega\right)  $. Therefore, by
Theorem \ref{AttrExist} the existence of the global attractor and its
compactness in $H_{0}^{1}\left(  \Omega\right)  $ and boundedness in
$L^{p}\left(  \Omega\right)  $ follow.

The proof of (\ref{distancias}) is analogous to that in Theorem 29 in
\cite{kapustyankasyanov}.
\end{proof}

\begin{definition}
The map $\gamma:\mathbb{R}\rightarrow L^{2}(\Omega)$ is called a complete
trajectory of $K_{r}^{+}$ if $\gamma(\cdot+h)|_{[0,\infty)}\in K_{r}^{+}$ for
any $h\in\mathbb{R}$. The set of all complete trajectories of $K_{r}^{+}$ will
be denoted by $\mathbb{F}_{r}$. Moreover, we write $\mathbb{K}_{r}$ as the set
of all complete trajectories which are bounded in $L^{2}(\Omega)$, and
$\mathbb{K}_{r}^{1}$ as the ones bounded in $H_{0}^{1}(\Omega)$.
\end{definition}

\begin{lemma}
\label{lemmaequality}Assume the conditions of Lemma \ref{lemma1}. Then the
sets defined above coincide, that is, $\mathbb{K}_{r}=\mathbb{K}_{r}^{1}$.
\end{lemma}

\begin{proof}
Let $\gamma(\cdot)\in\mathbb{K}_{r}$. Then there is $C$ such that $\left\Vert
\gamma\left(  t\right)  \right\Vert _{L^{2}}\leq C$ for any $t\in\mathbb{R}$.
Let $u_{\tau}\left(  \text{\textperiodcentered}\right)  =\gamma\left(
\text{\textperiodcentered}+\tau\right)  $ for any $\tau$, which is a regular
solution. Since $\dfrac{du}{dt}\in L^{2}(\varepsilon,T;L^{2}(\Omega)),\ $for
any $\varepsilon>0$, the equality (\ref{sellandyou1}) holds true. Therefore,
we can multiply the equation in (\ref{1}) by $u_{t}$ and apply again similar
arguments as in Theorem \ref{existenceweaksolutionu0L2sinderivadaversion2014}
to deduce that
\begin{equation}
\Vert u(t+r)\Vert_{H_{0}^{1}}^{2}\leq\frac{K_{1}\left(  T\right)  (1+\Vert
u(0)\Vert_{L^{2}}^{2})}{r}+K_{2}\left(  T\right)  \text{ for any }0<r<T.
\label{500}%
\end{equation}
Denote $B_{\gamma}=\cup_{t\in\mathbb{R}}\gamma(t)$. Therefore,
\[
B_{\gamma}\subset G_{r}(1,B_{\gamma})
\]
and (\ref{500}) implies that $B_{\gamma}$ is bounded in $H_{0}^{1}(\Omega)$,
so $\gamma(\cdot)\in\mathbb{K}_{r}^{1}$.\newline The other inclusion is obvious.
\end{proof}

\bigskip

In view of Corollary \ref{PropK4} and Theorem \ref{structureattractor}, the
global attractor is characterized in terms of bounded complete trajectories:
\begin{equation}%
\begin{split}
\mathcal{A}_{r}  &  =\{\gamma(0):\gamma(\cdot)\in\mathbb{K}_{r}\}=\{\gamma
(0):\gamma(\cdot)\in\mathbb{K}_{r}^{1}\}\\
&  =\bigcup_{t\in\mathbb{R}}\{\gamma(t):\gamma(\cdot)\in\mathbb{K}%
_{r}\}=\bigcup_{t\in\mathbb{R}}\{\gamma(t):\gamma(\cdot)\in\mathbb{K}_{r}%
^{1}\}.
\end{split}
\label{caracteattrac}%
\end{equation}

The set $\mathfrak{R}_{K_{r}^{+}}$ was defined in the previous section as the
set of fixed points of $K_{r}^{+}$, which means that $z\in\mathfrak{R}%
_{K_{r}^{+}}$ if the function $u\left(  \text{\textperiodcentered}\right)  $
defined by $u\left(  t\right)  =z,$ for all $t\geq0,$ belongs to $K_{r}^{+}.$
This set can be characterized as follows.

\begin{lemma}
\label{FixedCharacterization}Assume the conditions of Lemma \ref{lemma1}. Let
$\mathfrak{R}$ be the set of $z\in H^{2}\left(  \Omega\right)  \cap H_{0}%
^{1}\left(  \Omega\right)  $ such that
\begin{equation}
-a(\Vert z\Vert_{H_{0}^{1}}^{2})\Delta z=f(z)+h\text{ in }L^{2}\left(
\Omega\right)  . \label{Fixed}%
\end{equation}
Then $\mathfrak{R}_{K_{r}^{+}}=\mathfrak{R}.$
\end{lemma}

\begin{proof}
If $z\in\mathfrak{R}_{K_{r}^{+}}$, then $u\left(  t\right)  \equiv z\in
K_{r}^{+}$.Thus, $u\left(  \text{\textperiodcentered}\right)  $ satisfies
(\ref{EquationRegular}) and $\dfrac{du}{dt}=0$ in $L^{2}\left(  0,T;L^{2}%
\left(  \Omega\right)  \right)  $, so (\ref{Fixed}) is satisfied.

Let $z\in\mathfrak{R}$. Then the map $u\left(  t\right)  \equiv z$ satisfies
(\ref{Fixed}) for any $t\geq0$ and $\dfrac{du}{dt}=0$ in $L^{2}\left(
0,T;L^{2}\left(  \Omega\right)  \right)  $, so (\ref{EquationRegular}) holds true.
\end{proof}

\bigskip

The following result is proved exactly as Theorem \ref{boundinLinfty}.

\begin{theorem}
\label{boundinLinftyB}Assume the conditions of Lemma \ref{lemma1}. Then the
global attractor $\mathcal{A}$ is bounded in $L^{\infty}(\Omega)$, provided
that $h\in L^{\infty}(\Omega)$.
\end{theorem}

We are now ready to obtain the corresponding characterization of the global
attractor as in (\ref{caract1}).

\begin{theorem}
\label{theorem12}Let assume (\ref{Cont})-(\ref{3}), (\ref{7}), (\ref{h2}) and
one of the following conditions:

\begin{enumerate}
\item $f\in C^{1}(\mathbb{R})$ is such that $f^{\prime}(s)\leq\eta$ and $h\in
L^{\infty}\left(  \Omega\right)  ;$

\item $f\in C^{1}(\mathbb{R})$ is such that $f^{\prime}(s)\leq\eta$ and
$p\leq\frac{2n}{n-2}$ if $n\geq3;$

\item $p$ satisfies (\ref{Condp}).
\end{enumerate}

Then it holds that
\[
\mathcal{A}_{r}=M_{r}^{u}(\mathfrak{R})=M_{r}^{s}(\mathfrak{R}),
\]
where
\begin{equation}
M_{r}^{s}(\mathfrak{R})=\{z:\exists\gamma(\cdot)\in\mathbb{K}_{r}%
,\ \gamma(0)=z,\ \text{ }dist\text{ }_{L^{2}(\Omega)}(\gamma(t),\mathfrak{R}%
)\rightarrow0,\ t\rightarrow+\infty\}, \label{StableSet}%
\end{equation}%
\begin{equation}
M_{r}^{u}(\mathfrak{R})=\{z:\exists\gamma(\cdot)\in\mathbb{F}_{r}%
,\ \gamma(0)=z,\ \text{ }dist\text{ }_{L^{2}(\Omega)}(\gamma(t),\mathfrak{R}%
)\rightarrow0,\ t\rightarrow-\infty\}. \label{UnstableSet}%
\end{equation}

\end{theorem}

\begin{remark}
In the definition of $M_{r}^{u}(\mathfrak{R})$ we can replace $\mathbb{F}_{r}$
by $\mathbb{K}_{r}$. Also, as the global attractor $\mathcal{A}$ is compact in
$H_{0}^{1}\left(  \Omega\right)  $, in the definitions of $M_{r}%
^{s}(\mathfrak{R})$ and $M_{r}^{u}(\mathfrak{R}),$ it is equivalent to write
$H_{0}^{1}\left(  \Omega\right)  $ instead of $L^{2}\left(  \Omega\right)  .$
\end{remark}

\begin{proof}
We consider the function $E:\mathcal{A}_{r}\rightarrow\mathbb{R}$
\begin{equation}
E(y)=\frac{1}{2}A(\Vert y\Vert_{H_{0}^{1}}^{2})-\int_{\Omega}\mathcal{F}%
(y\left(  x\right)  )dx-\int_{\Omega}h\left(  x\right)  y\left(  x\right)  dx,
\label{EnergyFunction0}%
\end{equation}
where $A(r)=\int_{0}^{r}a(s)ds$. We observe that $E(y)$ is continuous in
$H_{0}^{1}\left(  \Omega\right)  $. Indeed, the maps $y\mapsto\frac{1}%
{2}A(\Vert y\Vert_{H_{0}^{1}}^{2}),\ y\mapsto\int_{\Omega}h\left(  x\right)
y\left(  x\right)  dx$ are obviously continuous in $H_{0}^{1}\left(
\Omega\right)  $. On the other hand, if $h\in L^{\infty}\left(  \Omega\right)
$, taking into account that $\mathcal{A}$ is bounded in $L^{\infty}\left(
\Omega\right)  $ by Theorem \ref{boundinLinftyB}, it follows that
\[
\left\vert \int_{\Omega}\mathcal{F}(y_{1})-\mathcal{F}(y_{2})dx\right\vert
=\left\vert \int_{\Omega}\int_{y_{2}\left(  x\right)  }^{y_{1}\left(
x\right)  }f(s)dsdx\right\vert \leq\int_{\Omega}C_{1}|y_{1}\left(  x\right)
-y_{2}\left(  x\right)  |dx\leq C_{2}\Vert y_{1}-y_{2}\Vert_{L^{2}},
\]
so $y\mapsto\int_{\Omega}\mathcal{F}(y\left(  x\right)  )dx$ is continuous. In
the other cases, making use of the embedding $H_{0}^{1}\left(  \Omega\right)
\subset L^{p}\left(  \Omega\right)  $ and Lebesgue's theorem the continuity of
$y\mapsto\int_{\Omega}\mathcal{F}(y\left(  x\right)  )dx$ follows as well.

Since $\dfrac{du}{dt}\in L^{2}\left(  \varepsilon,T;L^{2}\left(
\Omega\right)  \right)  $ for any $u\in K_{r}^{+}$ and $0<\varepsilon<T$, we
obtain the energy equality
\begin{equation}
\int_{s}^{t}\Vert\frac{d}{dr}\Vert u(r)\Vert_{L^{2}}^{2}%
dr+E(u(t))=E(u(s))\quad\text{ for all }t\geq s>0. \label{Energy}%
\end{equation}
Hence, $E\left(  u\left(  t\right)  \right)  $ is non-increasing and, by
(\ref{2}) and (\ref{4}), bounded from below. Thus, $E(u(t))\rightarrow l$, as
$t\rightarrow+\infty$, for some $l\in\mathbb{R}$.

Let $z\in\mathcal{A}_{r}$ and $\gamma\in\mathbb{K}_{r}$. We reason by
contradiction, so let suppose that there exists $\varepsilon>0$ and a sequence
$\gamma(t_{n}),$ $t_{n}\rightarrow+\infty$, such that
\[
\text{ }dist\text{ }_{L^{2}(\Omega)}(\gamma(t_{n}),\mathfrak{R})>\varepsilon.
\]
In view of Theorem \ref{existenceatracttorcase2}, $\mathcal{A}_{r}$ is compact
in $H_{0}^{1}(\Omega)$, so we can take a converging subsequence (relabeled the
same) such that $\gamma(t_{n})\rightarrow y$ in $H_{0}^{1}(\Omega)$, where
$t_{n}\rightarrow+\infty$. Since the function $E:H_{0}^{1}(\Omega
)\rightarrow\mathbb{R}$ is continuous, it follows that $E(y)=l$. We obtain a
contradiction by proving that $y\in\mathfrak{R}$. In view of Lemma
\ref{lemma1}, there exists $v\in K_{r}^{+}$ and a subsequence $v_{n}\left(
\text{\textperiodcentered}\right)  =\gamma(\cdot+t_{n})$ such that $v(0)=y$
and $v_{n}(t)\rightarrow v(t)=z$ in $H_{0}^{1}(\Omega)$. Thus, $E(v_{n}%
(t))\rightarrow E(z)$ implies that $E(z)=l$. Also, $v(\cdot)$ satisfies the
energy equality for all $0\leq s\leq t$, so that
\[
l+\int_{0}^{t}\Vert v_{r}\Vert_{L^{2}}^{2}dr=E(z)+\int_{0}^{t}\Vert v_{r}%
\Vert_{L^{2}}^{2}dr=E(v(0))=E(y)=l.
\]
Therefore, $\dfrac{dv}{dt}(t)=0$ for a.a. $t$, and then by Lemma
\ref{FixedCharacterization} we have $y\in\mathfrak{R}_{K_{r}^{+}}%
\mathfrak{=R}$. As a consequence, $\mathcal{A}_{r}\subset M_{r}^{s}%
(\mathfrak{R})$. The converse inclusion follows from \ref{caracteattrac}.

For the second equality we observe that for any $\gamma\in\mathbb{F}_{r}$ the
energy equality (\ref{Energy}) is satisfied for all $-\infty<s\leq t.$ Let
$z\in\mathcal{A}_{r}$ and let $\gamma\in\mathbb{K}_{r}=\mathbb{K}_{r}^{1}$
(cf. Lemma \ref{lemmaequality}) be such that $\gamma(0)=z$. Since the second
term of the energy function is bounded from above by (\ref{4}), $E(\gamma
(t))\rightarrow l$, as $t\rightarrow-\infty$, for some $l\in\mathbb{R}$. We
reason as before, so let suppose that there exists $\varepsilon>0$ and a
sequence $\gamma(-t_{n})$, $t_{n}\rightarrow\infty$, such that
\[
\text{ }dist\text{ }_{L^{2}(\Omega)}(\gamma(-t_{n}),\mathfrak{R}%
)>\varepsilon,
\]
and we have that $\gamma(-t_{n})\rightarrow y$ in $H_{0}^{1}(\Omega)$,
$E(y)=l$. Moreover, for a fixed $t>0$, there exists $v\in K_{r}^{+}$ and a
subsequence of $v_{n}(\cdot)=\gamma(\cdot-t_{n})$ (relabeled the same) such
that $v(0)=y$ and $v_{n}(t)\rightarrow v(t)=z$ in $H_{0}^{1}(\Omega)$.
Therefore, $E(v_{n}(t))\rightarrow E(z)$ implies that $E(z)=l$ and reasoning
as before, we get a contradiction since it follows that $y\in\mathfrak{R}$.
Hence, $\mathcal{A}_{r}\subset M_{r}^{u}(\mathfrak{R})$ and the converse
inclusion follows from \ref{caracteattrac}.
\end{proof}

\bigskip

If we assume the conditions of Section \ref{SectionAttr1}, then $G_{r}=T_{r}$,
where $T_{r}$ is the semigroup defined by (\ref{Semigroup}). Then, we can
improve the regularity of the global attractor of $T_{r}$.

\begin{corollary}
\label{ConvergH1S}Let the conditions of Theorem \ref{ExistAttr2} hold. Then
the global attractor $\mathcal{A}_{r}$ of the semigroup $T_{r}$ is compact in
$H_{0}^{1}\left(  \Omega\right)  $ and the convergence takes place in the
topology of $H_{0}^{1}\left(  \Omega\right)  $, that is,%
\[
dist_{H_{0}^{1}\left(  \Omega\right)  }(T_{r}(t,B),\mathcal{A})\rightarrow
0,\text{ as }t\rightarrow+\infty,
\]
for any set $B$ bounded in $L^{2}\left(  \Omega\right)  .$
\end{corollary}

\begin{proof}
Since the semigroup $T_{r}$ coincides with the semiflow $G_{r}$, the results
follows from Theorem \ref{existenceatracttorcase2}.
\end{proof}

In addition, from Theorem \ref{theorem12} the characterization of the global
attractor follows.

\begin{proposition}
\label{CharacAttrUniqueness}Let the conditions of Theorem \ref{ExistAttr2}
hold. Assume that one of the following conditions hold:

\begin{enumerate}
\item $h\in L^{\infty}\left(  \Omega\right)  $;

\item $p\leq\frac{2n}{n-2}$ if $n\geq3.$
\end{enumerate}

Then the global attractor $\mathcal{A}_{r}$ can be caracterized as follows
\begin{equation}
\mathcal{A}_{r}=M_{r}^{u}(\mathfrak{R})=M_{r}^{s}(\mathfrak{R}),
\label{caract1}%
\end{equation}
where $M_{r}^{s}(\mathfrak{R}),\ M_{r}^{u}(\mathfrak{R})$ are defined in
(\ref{StableSet})-(\ref{UnstableSet}).
\end{proposition}

\begin{remark}
This proposition also follows from \cite[p.160]{babin}, as a Lyapunov function
exists inside the global attractor.
\end{remark}

\subsection{Strong solutions}

We split this part into two cases.

\subsubsection{Attractor in the phase space $H_{0}^{1}\left(  \Omega\right)
\cap L^{p}\left(  \Omega\right)  $}

We consider the metric space $X=H_{0}^{1}\left(  \Omega\right)  \cap
L^{p}\left(  \Omega\right)  $ endowed with the induced topology of the space
$H_{0}^{1}\left(  \Omega\right)  $.

If we assume that $f\in C^{1}(\mathbb{R})$ is such that $f^{\prime}(s)\leq
\eta$ and conditions (\ref{Cont})-(\ref{3}), (\ref{h2}) as well, then we known
by Theorem \ref{existencestrongsolutionu0H01Lp} that for any $u_{0}\in X$
there exists al least one strong solution $u\left(  \text{\textperiodcentered
}\right)  $,

Let
\[
\mathcal{R}=K_{s}^{+}:=\{u(\cdot):u\text{ is a strong solution of (\ref{1}%
)}\}.
\]
We define the (possibly multivalued) map $G_{s}:\mathbb{R}^{+}\times
X\rightarrow P(X)$ by
\[
G_{s}(t,u_{0})=\{u(t):u\in K_{s}^{+}\text{ and }u(0)=u_{0}\}.
\]
With respect to the axiomatic properties $\left(  K1\right)  -\left(
K4\right)  $ given above, property $\left(  K1\right)  $ is obviously true,
and $\left(  K2\right)  -\left(  K3\right)  $ can be proved easily using
equality (\ref{EquationRegular}). Therefore, $G_{s}$ is a strict multivalued
semiflow by the results of Section \ref{Abstract}.

We shall obtain a similar result as in Lemma \ref{lemma1}.

\begin{lemma}
\label{lemma4} Let assume that $f\in C^{1}(\mathbb{R})$ is such that
$f^{\prime}(s)\leq\eta$ and conditions (\ref{Cont})-(\ref{3}), (\ref{h2}).
Given $\{u^{n}\}\subset K_{s}^{+}$ a sequence such that $u^{n}(0)\rightarrow
u_{0}$ weakly in $H_{0}^{1}(\Omega)\cap L^{p}\left(  \Omega\right)  $, there
exists a subsequence of $\{u^{n}\}$ (relabeled the same) and $u\in K_{s}^{+}$,
satisfying $u(0)=u_{0}$, such that
\begin{align*}
u^{n}(t)  &  \rightarrow u(t)\mathit{\ }\text{\textit{ in }}H_{0}^{1}%
(\Omega),\ \forall t>0,\\
u^{n}\left(  t\right)   &  \rightharpoonup u\left(  t\right)  \text{ in }%
L^{p}\left(  \Omega\right)  ,\ \forall t\geq0.
\end{align*}

\end{lemma}

\begin{proof}
Since $\dfrac{du^{n}}{dt}\in L^{2}\left(  0,T;L^{2}\left(  \Omega\right)
\right)  $, we can use (\ref{sellandyou1}) and multiplying (\ref{1}) by
$u_{t}$ and integrating between $s$ and $t$ we obtain
\[
\int_{s}^{t}\Vert\frac{d}{dr}\Vert u(r)\Vert_{L^{2}}^{2}%
dr+E(u(t))=E(u(s))\quad\text{ for all }t\geq s\geq0,
\]
where $E$ was defined in (\ref{EnergyFunction0}). Therefore, by (\ref{2}) and
(\ref{4}) we have that%
\begin{equation}
\int_{0}^{t}\Vert\frac{d}{dr}u(r)\Vert_{L^{2}}^{2}dr+\frac{m}{4}\Vert
u(t)\Vert_{H_{0}^{1}}^{2}+\widetilde{\alpha}_{1}\Vert u(t)\Vert_{L^{p}}%
^{p}\leq\frac{1}{2}A(\Vert u(0)\Vert_{H_{0}^{1}}^{2})+\widetilde{\alpha}%
_{2}\Vert u(0)\Vert_{L^{p}}^{p}+K_{1}\left\Vert u\left(  0\right)  \right\Vert
_{L^{2}}^{2}+K_{2} \label{69}%
\end{equation}
holds for all $t>0$.
\end{proof}

On the other hand, multiplying by $-\Delta u$, integrating over $(0,T)$ and
using (\ref{69}) it follows that%
\begin{equation}
\frac{1}{2}\Vert u(T)\Vert_{H_{0}^{1}}^{2}+\frac{m}{2}\int_{0}^{T}\Vert\Delta
u(s)\Vert_{L^{2}}^{2}ds\leq\eta\int_{0}^{T}\Vert u(s)\Vert_{H_{0}^{1}}%
^{2}ds+\frac{1}{2}\Vert u(0)\Vert_{H_{0}^{1}}^{2}+K_{3}\leq K_{4}\left(
T\right)  , \label{70}%
\end{equation}
for all $T>0$.

Thus, the sequence $u^{n}$ is bounded in $L^{\infty}(0,T;H_{0}^{1}(\Omega)\cap
L^{p}\left(  \Omega\right)  )\cap L^{2}(0,T;D(A))$ and $\dfrac{du^{n}}%
{dt},\ f\left(  u^{n}\right)  $ are bounded in $L^{2}(0,T;L^{2}(\Omega))$, for
all $T>0$. Therefore,
\[
u^{n}\overset{\ast}{\rightharpoonup}u\text{ in }L^{\infty}(0,T;H_{0}%
^{1}(\Omega)\cap L^{p}\left(  \Omega\right)  ),
\]%
\[
u^{n}\rightharpoonup u\text{ in }L^{2}(0,T;D(A)),
\]%
\[
u_{t}^{n}\rightharpoonup u_{t}\text{ in }L^{2}(0,T;L^{2}(\Omega)),.
\]
Arguing in a similar way as in Theorem \ref{existenceweakregularsolutionu0L2}
we have%
\[
u_{n}\rightarrow u\text{ in }L^{2}(0,T;H_{0}^{1}(\Omega)),
\]%
\[
u_{n}(t,x)\rightarrow u(t,x)\text{ a.e. on }(0,T)\times\Omega,
\]%
\[
f\left(  u^{n}\right)  \rightharpoonup f\left(  u\right)  \text{ in }%
L^{2}(0,T;L^{2}(\Omega)),
\]

\[
a(\Vert u_{n}\Vert_{H_{0}^{1}}^{2})\Delta u_{n}\rightharpoonup a(\Vert
u\Vert_{H_{0}^{1}}^{2})\Delta u\text{ in }L^{2}(0,T;L^{2}(\Omega)).
\]
Hence, we can pass to the limit and obtain that $u\in K_{s}^{+}$. Following
the same lines of Theorem \ref{existencestrongsolutionu0H01Lp} we check that
$u\left(  0\right)  =u_{0}.$

Moreover, arguing as in Lemma \ref{lemma1} we obtain
\[
u^{n}(t)\rightarrow u(t)\text{ in }H_{0}^{1}(\Omega)\text{ for all }t>0.
\]
Finally, by the Ascoli-Arzel\`{a} theorem we deduce%
\[
u^{n}\rightarrow u\text{ in }C([0,T],L^{2}(\Omega))
\]
and combining this with (\ref{69}) we infer that%
\[
u^{n}\left(  t\right)  \rightharpoonup u\left(  t\right)  \text{ in }%
L^{p}\left(  \Omega\right)  \text{ }\forall t\geq0.
\]

\begin{corollary}
Assume the conditions of Lemma \ref{lemma4}. Then the multivalued semiflow
$G_{s}$ has compact values.
\end{corollary}

\begin{proposition}
\label{absorbingset3}Assume the conditions of Lemma \ref{lemma4}. Then there
exists an absorbing set $B_{1}$ for $G_{s}$, which is compact in $H_{0}%
^{1}\left(  \Omega\right)  $ and bounded in $L^{p}\left(  \Omega\right)  $.
\end{proposition}

\begin{proof}
The proof follows the same lines of that in Proposition \ref{absorbingset} but
using Lemma \ref{lemma4}.
\end{proof}

\bigskip

From these results we obtain the existence of the global attractor.

\begin{theorem}
\label{existenceatracttorcase24} Assume the conditions of Lemma \ref{lemma4}.
Then the multivalued semiflow $G_{s}$ possesses a global compact invariant
attractor $\mathcal{A}_{s}$, which is bounded in $L^{p}\left(  \Omega\right)
$.
\end{theorem}

\begin{proof}
In this case we cannot apply Theorem \ref{AttrExist}, because we do not know
if the map $u_{0}\mapsto G_{s}(t.u_{0})$ is upper semicontinuous. We state
that
\[
\mathcal{A}_{s}=\omega\left(  B_{1}\right)  =\{y:\exists t_{n}\rightarrow
+\infty,\ y_{n}\in G_{s}\left(  t_{n},B_{1}\right)  \text{ such that }%
y_{n}\rightarrow y\text{ in }X\}
\]
is a global compact attractor. The set $\omega\left(  B_{1}\right)  $ is
non-empty, compact and the minimal closed set attracting $B_{1}$ \cite[Theorem
6]{CLMV03}. Since $B_{1}$ is absorbing, $\omega\left(  B_{1}\right)  $
attracts any bounded set $B$.

It remains to prove that it is invariant.

First, we prove that it is negatively invariant. Let $y\in\mathcal{A}_{s}$ and
$t>0$ be arbitrary. We take a sequence $y_{n}\in G_{s}\left(  t_{n}%
,B_{1}\right)  $ such that $y_{n}\rightarrow y$, $t_{n}\rightarrow+\infty$.
Since $G_{s}\left(  t_{n},B_{1}\right)  \subset G_{s}(1,G_{s}(t_{n}-1,B_{1}%
))$, there are $x_{n}\in G_{s}(t_{n}-1,B_{1})$ and $u_{n}\in K_{s}^{+}$ such
that $u_{n}\left(  0\right)  =x_{n}$. As for $n$ large $G_{s}(t_{n}%
-1,B_{1})\subset B_{1}$, the sequence $\{x_{n}\}$ \ is bounded in
$L^{p}\left(  \Omega\right)  $ and relatively compact in $H_{0}^{1}\left(
\Omega\right)  $. Hence, up to a subsequence $x_{n}\rightarrow x$ weakly in
$L^{p}\left(  \Omega\right)  $ and strongly in $H_{0}^{1}\left(
\Omega\right)  .$ We deduce by Lemma \ref{lemma4} the existence of a
subsequence of $\{u_{n}\}$ (relabeled the same) and $u\in K_{s}^{+}$,
satisfying $u(0)=x$, such that $u_{n}\left(  t\right)  \rightarrow u\left(
t\right)  $ weakly in $L^{p}\left(  \Omega\right)  $ and strongly in
$H_{0}^{1}\left(  \Omega\right)  $. Thus, $y=u\left(  t\right)  \in
G_{s}\left(  t,x\right)  \subset G_{s}\left(  t,\mathcal{A}_{s}\right)  .$

Second, we prove that it is positively invariant. As the semiflow $G_{s}$ is
strict and $\mathcal{A}_{s}\subset G(\tau,\mathcal{A}_{s})$ for any $\tau
\geq0$, this follows from%
\[
dist_{X}\left(  G_{s}\left(  t,\mathcal{A}_{s}\right)  ,\mathcal{A}%
_{s}\right)  \leq dist_{X}\left(  G_{s}\left(  t,G(\tau,\mathcal{A}%
_{s})\right)  ,\mathcal{A}_{s}\right)  =dist_{X}\left(  G_{s}\left(
t+\tau,\mathcal{A}_{s}\right)  ,\mathcal{A}_{s}\right)  \underset{\tau
\rightarrow+\infty}{\rightarrow}0.
\]

\end{proof}

\begin{lemma}
\label{AttrEq}Assume the conditions of Lemma \ref{lemma4} and (\ref{7}). Then
$\mathcal{A}_{s}=\mathcal{A}_{r}$, where $\mathcal{A}_{r}$ is the global
attractor in Theorem \ref{existenceatracttorcase2}.
\end{lemma}

\begin{proof}
Since $G_{s}\left(  t,u_{0}\right)  \subset G_{r}\left(  t,u_{0}\right)  $ for
all $u_{0}\in X$, it is clear that $\mathcal{A}_{r}$ is a compact attracting
set. Hence, the minimality of the global attractor gives $\mathcal{A}%
_{s}\subset\mathcal{A}_{r}$.

Let $z\in\mathcal{A}_{r}$. Since $z=\gamma\left(  0\right)  $, where
$\gamma\in\mathbb{K}_{r}^{1}$, and $\gamma\mid_{\lbrack s,+\infty)}$ is a
strong solution of (\ref{1}) for any $s\in\mathbb{R},$ we get that $z\in
G_{s}(t_{n},\gamma\left(  -t_{n}\right)  )$ for $t_{n}\rightarrow+\infty$.
Hence,%
\[
dist\left(  z,\mathcal{A}_{s}\right)  \leq dist\left(  G_{s}(t_{n}%
,\gamma\left(  -t_{n}\right)  ),\mathcal{A}_{s}\right)  \rightarrow0\text{ as
}n\rightarrow\infty,
\]
so $z\in\mathcal{A}_{s}.$
\end{proof}

\bigskip

We define a complete trajectory of $G_{s}$ as a map $\gamma:\mathbb{R}%
\rightarrow X$ that satisfies $\gamma(\cdot+h)|_{[0,\infty)}\in K_{s}^{+}$ for
any $h\in\mathbb{R}$. The set of all complete trajetories of $K_{s}^{+}$ will
be denoted by $\mathbb{F}_{s}.$ Let $\mathbb{K}_{s}$ be the set of all
complete trajectories which are bounded in $H_{0}^{1}(\Omega)$.

In view of Theorem \ref{structureattractor}, the global attractor is
characterized in terms of bounded complete trajectories:
\begin{equation}
\mathcal{A}_{s}=\{\gamma(0):\gamma(\cdot)\in\mathbb{K}_{s}\}=\bigcup
_{t\in\mathbb{R}}\{\gamma(t):\gamma(\cdot)\in\mathbb{K}_{s}\}.
\label{caracteattrac2}%
\end{equation}

In the same way as in Lemma \ref{FixedCharacterization} we obtain that
$\mathfrak{R}_{K_{s}^{+}}=\mathfrak{R.}$

Reasoning as in Theorem \ref{boundinLinfty} we obtain the following result.

\begin{theorem}
\label{boundinLinftyC}Assume the conditions of Lemma \ref{lemma4}. Then the
global attractor $\mathcal{A}$ is bounded in $L^{\infty}(\Omega)$, provided
that $h\in L^{\infty}(\Omega)$.
\end{theorem}

Following the same procedure of Theorem \ref{theorem12} we can prove an
analogous characterization of the global attractor.

\begin{theorem}
\label{structureattractorStrong}Assume the conditions of Lemma \ref{lemma4}
and that one of following assumptions holds:

\begin{enumerate}
\item $h\in L^{\infty}\left(  \Omega\right)  ;$

\item $p\leq\frac{2n}{n-2}$ if $n\geq3;$
\end{enumerate}

Then it holds that
\[
\mathcal{A}_{s}=M_{s}^{u}(\mathfrak{R})=M_{s}^{s}(\mathfrak{R}),
\]
where
\begin{equation}
M_{s}^{s}(\mathfrak{R})=\{z:\exists\gamma(\cdot)\in\mathbb{K}_{s}%
,\ \gamma(0)=z,\ \text{ }dist\text{ }_{H_{0}^{1}(\Omega}(\gamma
(t),\mathfrak{R})\rightarrow0,\ t\rightarrow\infty\}, \label{StableSet2}%
\end{equation}%
\begin{equation}
M_{s}^{u}(\mathfrak{R})=\{z:\exists\gamma(\cdot)\in\mathbb{F}_{s}%
,\ \gamma(0)=z,\ \text{ }dist_{H_{0}^{1}(\Omega}(\gamma(t),\mathfrak{R}%
)\rightarrow0,\ t\rightarrow-\infty\}. \label{UnstableSet2}%
\end{equation}

\end{theorem}

\begin{remark}
In the definition of $M_{s}^{u}(\mathfrak{R})$ we can replace $\mathbb{F}_{r}$
by $\mathbb{K}_{r}$.
\end{remark}

Let us consider now the particular situation when $G_{s}$ is single-valued
semigroup. If we assume conditions (\ref{Cont})-(\ref{3}), (\ref{1.36}),
(\ref{h2}), $f\in C^{1}(\mathbb{R})$ and$\ f^{\prime}(s)\leq\eta$, then by
Theorems \ref{existencestrongsolutionu0H01Lp} and \ref{uniqueness} for any
$u_{0}\in H_{0}^{1}\left(  \Omega\right)  \cap L^{p}\left(  \Omega\right)  $
there exists a unique strong solution $u\left(  \text{\textperiodcentered
}\right)  $. Then we can define the following semigroup $T_{s}:\mathbb{R}%
^{+}\times X\rightarrow X:$
\[
T_{s}(t,u_{0})=u(t),
\]
where $u\left(  \text{\textperiodcentered}\right)  $ is the unique strong
solution to (\ref{1}). We recall also that $u\in C([0,T],H_{0}^{1}\left(
\Omega)\right)  \cap C_{w}\left(  [0,T],L^{p}\left(  \Omega\right)  \right)  $
for any $T>0.$ Also, by Lemma \ref{lemma4} if $u_{0}^{n}\rightarrow u_{0}$
weakly in $H_{0}^{1}\left(  \Omega\right)  \cap L^{p}\left(  \Omega\right)  $,
then $T_{s}(t,u_{0}^{n})\rightarrow T\left(  t,u_{0}\right)  $ in $H_{0}%
^{1}\left(  \Omega\right)  $ for all $t>0.$

Since $T_{s}=G_{s}$, by Theorems \ref{existenceatracttorcase24},
\ref{boundinLinftyC}, \ref{structureattractorStrong} and Lemma \ref{AttrEq} we
obtain the following results.

\begin{theorem}
\label{ExistAttrStrongUniq}We assume conditions (\ref{Cont})-(\ref{3}),
(\ref{1.36}), (\ref{h2}), $f\in C^{1}(\mathbb{R})$ and$\ f^{\prime}(s)\leq
\eta$. Then the semigroup $T_{s}$ posseses a global invariant attractor
$\mathcal{A}_{s}$, which is compact in $X$ and bounded in $L^{p}\left(
\Omega\right)  $.
\end{theorem}

\begin{lemma}
Under the conditions of Theorem \ref{ExistAttrStrongUniq} and (\ref{7}),
$\mathcal{A}_{s}=\mathcal{A}_{r}$, where $\mathcal{A}_{r}$ is the attractor of
Theorem \ref{ExistAttr2}.
\end{lemma}

\begin{theorem}
Assume the conditions of Theorem \ref{ExistAttrStrongUniq}. Then the global
attractor $\mathcal{A}_{s}$ is bounded in $L^{\infty}(\Omega)$ provided that
$h\in L^{\infty}(\Omega)$.
\end{theorem}

As before, we denote by $\mathfrak{R}$ the set of fixed points of $T$. Also,
the global attractor is the union of all boundedcomplete trajectories%
\[
\mathcal{A}_{s}=\{\phi\left(  0\right)  :\phi\text{ is a bounded complete
trajectory of }T_{s}\}.
\]

\begin{theorem}
We assume conditions (\ref{Cont})-(\ref{3}), (\ref{1.36}),\ (\ref{h2}), $f\in
C^{1}(\mathbb{R})$,$\ f^{\prime}(s)\leq\eta$ and one of the following assumptions:
\end{theorem}

\begin{proposition}
\begin{enumerate}
\item $h\in L^{\infty}\left(  \Omega\right)  $;

\item $p\leq\frac{2n}{n-2}$ if $n\geq3.$
\end{enumerate}

Then the global attractor $\mathcal{A}_{s}$ can be caracterized as follows
\[
\mathcal{A}_{s}=M_{s}^{u}(\mathfrak{R})=M_{s}^{s}(\mathfrak{R}),
\]
where the sets $M_{s}^{u}(\mathfrak{R}),\ M_{s}^{s}(\mathfrak{R})$ are defined
in (\ref{StableSet2})-(\ref{UnstableSet2}).
\end{proposition}

In this case we can obtain additionally that the attractor is bounded in
$H^{2}\left(  \Omega\right)  .$

\begin{proposition}
Assume the conditions of Theorem \ref{ExistAttrStrongUniq} and also that $a\in
C^{1}(\mathbb{R},\mathbb{R})$, $a^{\prime}\left(  s\right)  \geq0$. Then
$\mathcal{A}_{s}$ is bounded in $H^{2}\left(  \Omega\right)  .$
\end{proposition}

\begin{proof}
The proof follows the same lines as in Proposition \ref{boundinH2}, so we omit it.
\end{proof}

\subsubsection{Attractor in the phase space $H_{0}^{1}\left(  \Omega\right)
$}

We consider now the metric space $X=H_{0}^{1}\left(  \Omega\right)  $.

If we assume that conditions (\ref{Cont})-(\ref{3}), (\ref{h2}) and
(\ref{Condp}), then by Theorem \ref{existenciau0paper2014} and the argument of
Section \ref{SectionNonuniqueness} for any $u_{0}\in X$ there exists al least
one strong solution $u\left(  \text{\textperiodcentered}\right)  $ and every
strong solution satisfies that $\dfrac{du}{dt},f\left(  u\right)  \in
L^{2}\left(  0,T;L^{2}\left(  \Omega\right)  \right)  $, for all $T>0$.

As in the previous section for
\[
\mathcal{R}=K_{s}^{+}:=\{u(\cdot):u\text{ is a strong solution of (\ref{1}%
)}\}
\]
we define the (possibly multivalued) map $G_{s}:\mathbb{R}^{+}\times
X\rightarrow P(X)$ by
\[
G_{s}(t,u_{0})=\{u(t):u\in K_{s}^{+}\text{ and }u(0)=u_{0}\},
\]
which satisfies $\left(  K1\right)  -\left(  K3\right)  $ and then it is a
strict multivalued semiflow.

We obtain the same results as in the previous section.

\begin{lemma}
\label{lemma4B} Let assume (\ref{Cont})-(\ref{3}), (\ref{h2}) and
(\ref{Condp}). Given $\{u^{n}\}\subset K_{s}^{+}$ a sequence such that
$u^{n}(0)\rightarrow u_{0}$ weakly in $H_{0}^{1}(\Omega)\cap L^{p}\left(
\Omega\right)  $, there exists a subsequence of $\{u^{n}\}$ (relabeled the
same) and $u\in K_{s}^{+}$, satisfying $u(0)=u_{0}$, such that
\begin{align*}
u^{n}(t)  &  \rightarrow u(t)\mathit{\ }\text{\textit{ in }}H_{0}^{1}%
(\Omega),\ \forall t>0,\\
u^{n}\left(  t\right)   &  \rightharpoonup u\left(  t\right)  \text{ in }%
L^{p}\left(  \Omega\right)  ,\ \forall t\geq0.
\end{align*}

\end{lemma}

\begin{proof}
Arguing as in Lemma \ref{lemma4} we obtain inequality (\ref{69}) and combining
it with (\ref{Acotacionfu}) the boundedness of $f\left(  u^{n}\right)  $ in
$L^{2}\left(  0,T;L^{2}\left(  \Omega\right)  \right)  $ follows for any
$T>0$. Hence, the equality
\[
a\left(  \left\Vert u\right\Vert _{H_{0}^{1}}^{2}\right)  \Delta
u=\frac{du^{n}}{dt}-f\left(  u^{n}\right)  -h
\]
and (\ref{2}) imply that $u^{n}$ is bounded in $L^{2}\left(  0,T;D(A)\right)
.$

The rest of the proof follows the same lines as in Lemma \ref{lemma4}.
\end{proof}

\bigskip

The proofs of the other results remain unchanged with respect to the same ones
in the previous section, so we omit them.

\begin{corollary}
Assume the conditions of Lemma \ref{lemma4B}. Then the set $K_{s}^{+}$
satisfies condition $(K4)$.
\end{corollary}

\begin{proposition}
Assume the conditions of Lemma \ref{lemma4B}. Then the multivalued semiflow
$G_{s}$ is upper semicontinuous for all $t\geq0$ with compact values.
\end{proposition}

\begin{proposition}
Assume the conditions of Lemma \ref{lemma4B}. Then there exists an absorbing
set $B_{1}$ for $G_{s}$, which is compact in $H_{0}^{1}\left(  \Omega\right)
$ and bounded in $L^{p}\left(  \Omega\right)  $.
\end{proposition}

\begin{theorem}
Assume the conditions of Lemma \ref{lemma4B}. Then the multivalued semiflow
$G_{s}$ possesses a global compact invariant attractor $\mathcal{A}_{s}$,
which is bounded in $L^{p}\left(  \Omega\right)  $.
\end{theorem}

\begin{proof}
The result follows from the previous results and Theorem \ref{AttrExist}.
\end{proof}

\begin{lemma}
Assume the conditions of Lemma \ref{lemma4B} and (\ref{7}). Then
$\mathcal{A}_{s}=\mathcal{A}_{r}$, where $\mathcal{A}_{r}$ is the global
attractor in Theorem \ref{existenceatracttorcase2}.
\end{lemma}

The definitions of complete trajectories are the same as in the previous
section. The global attractor is characterized in terms of bounded complete
trajectories:
\[
\mathcal{A}_{s}=\{\gamma(0):\gamma(\cdot)\in\mathbb{K}_{s}\}=\bigcup
_{t\in\mathbb{R}}\{\gamma(t):\gamma(\cdot)\in\mathbb{K}_{s}\}.
\]
Also, as in the previous section, $\mathfrak{R}_{K_{s}^{+}}=\mathfrak{R.}$

\begin{theorem}
Assume the conditions of Lemma \ref{lemma4B}. Then the global attractor
$\mathcal{A}$ is bounded in $L^{\infty}(\Omega)$ provided that $h\in
L^{\infty}(\Omega)$.
\end{theorem}

\begin{theorem}
Assume the conditions of Lemma \ref{lemma4B}. It holds
\[
\mathcal{A}_{s}=M_{s}^{u}(\mathfrak{R})=M_{s}^{s}(\mathfrak{R}),
\]
where $M_{s}^{u}(\mathfrak{R}),\ M_{s}^{s}(\mathfrak{R})$ are defined in
(\ref{UnstableSet2})-(\ref{StableSet2}).
\end{theorem}

\bigskip

\textbf{Acknowledgments.}

This work has been partially supported by Spanish Ministry of Economy and
Competitiveness and FEDER, projects MTM2015-63723-P and MTM2016-74921-P.
{\Large \textquestiondown Proyectos de la Junta?}

\end{document}